\documentclass{amsart}
\usepackage{amsmath,amssymb,graphicx}

\newcounter{sub}
\newenvironment{sub}%
{\begin{list}{(\arabic{sub})}{\usecounter{sub}%
\setlength{\leftmargin}{2em}}}{\end{list}}

\newcommand{\ch}{\mathop{\mathrm{char}}\nolimits}
\newcommand{\Gal}{\mathop{\mathrm{Gal}}\nolimits}
\newcommand{\Pic}{\mathop{\mathrm{Pic}}\nolimits}
\newcommand{\Div}{\mathop{\mathrm{Div}}\nolimits}
\newcommand{\Aut}{\mathop{\mathrm{Aut}}\nolimits}
\newcommand{\Br}{\mathop{\mathrm{Br}}\nolimits}
\def\sep{\mathrm{sep}}
\def\id{\mathrm{id}}
\def\im{\mathrm{Im}}
\def\ker{\mathrm{Ker}}

\makeatletter
    
    \@addtoreset{equation}{section}
\makeatother

\theoremstyle{plain}
\newtheorem{thm}{Theorem}[section]
\newtheorem{cor}{Corollary}[section]
\newtheorem{defn}{Definition}[section]

\newtheorem{prop}{Proposition}[section]
\newtheorem{exmp}{Example}[section]

\theoremstyle{remark}
\newtheorem{rem}{Remark}[section]

\title{Rationality problem of conic bundles}
\author{Aiichi Yamasaki}
\address{Department of Mathematics, Graduate school of Science, 
Kyoto University, Japan}
\thanks{{\it Key words and phrases.} Rationality problem,
conic bundle, Ch\^atelet surface, Picard group. \\
This work was partially supported by KAKENHI (24540019).
Some part of this work was done during the author visited 
National Taiwan University (Department of Mathematics), 
whose support is gratefully acknowledged.}

\begin{document}

\begin{abstract}
Let $k$ be a field with $\ch k \not= 2$, $X$ be an affine surface
defined by the equation $z^2=P(x)y^2+Q(x)$ where $P(x),Q(x) \in
k[x]$ are separable polynomials. We will investigate the
rationality problem of $X$ in terms of the polynomials $P(x)$ and
$Q(x)$. The rationality of the conic bundle $X$ over $\mathbb{P}_k^1$
was studied by Iskovskikh \cite{Isk67}, \cite{Isk70},
\cite{Isk72}, but he formulated his results in geometric
language. This paper aims to give an algebraic counterpart.
\end{abstract}

\maketitle

\section{Introduction}

Throughout this paper, $k$ is a field with $\ch k \not= 2$. It is
not assumed that $k$ is algebraically closed; in fact, the most
interesting results of this paper is the case when $k$ is a
non-closed field.

Let $K$ be a field extension of $k$. We will say that $K$ is
$k$-rational if $K$ is isomorphic to the rational function field
$k(x_1, x_2, \dots, x_n)$ over $k$ with variables $x_1,\dots,x_n$ for some positive integer $n$. An
irreducible algebraic variety $X$ defined over $k$ is called
$k$-rational if its function field $k(X)$ is $k$-rational.

Iskovskikh studied the rationality of conic bundles and obtained
the following result \cite{Isk67}, \cite{Isk70}, \cite{Isk72}.

\begin{thm}[Iskovskikh] \label{Isk}
Let $X$ be a fibred rational $k$-surface as a standard conic
bundle $\pi: X \rightarrow \mathbb{P}_k^1$. If $X$ has at least
four degenerate geometric fibres, then $X$ is not $k$-rational.
\end{thm}

The function field of such a conic bundle is isomorphic to $k(x,y,z)$
with the relation
\begin{equation} \label{chanew}
z^2=Q(x)y^2+P(x), \  P,Q \in k[x],
\end{equation}
where $k(x,y)$ is the rational function field over $k$ with two variables $x,y$.

Thus Iskovskikh's theorem (Theorem \ref{Isk}) is equivalent to the rationality problem
of the field $K:=k(x,y,z)$ with the relation defined by
(\ref{chanew}). In this paper, we will give a necessary and
sufficient condition for the rationality of $K$ in terms of the
polynomials $P$ and $Q$, assuming that both $P$ and $Q$ are
separable polynomials. In this sense, our results may be regarded
as an algebraic counterpart of Iskovskikh's theorem.

In Subsection \ref{main} and Section \ref{chCh} of this paper, we will consider the case where
$\deg Q(x)=0$, i.e. $Q(x)=a \in k$. The case where $\deg Q(x) \ge
1$ will be discussed in Subsection \ref{maincha} and Section \ref{chCo}.

\subsection{Main result of generalized Ch\^atelet surfaces.} \label{main}

First of all, let $K=k(x,y,z)$ be a field defined by the equation

\begin{equation} \label{cha}
z^2=ay^2+P(x), \quad a \in k, \, P(x) \in k[x].
\end{equation}

\begin{rem}
The surface $X$ defined by $(\ref{cha})$ is called a Ch\^atelet
surface when $\deg P=3$ or $4$, which was studied by Ch\^atelet
\cite{Cha59}. Thus we will call the surface $X$ a generalized
Ch\^atelet surface when $P$ is any non-zero polynomial in $k[x]$.
The function field of $X$ is the field $K$ defined by $(\ref{cha})$.
\end{rem}

\bigskip
Let $K$ be the function field of a generalized Ch\^atelet surface
defined by the equation (\ref{cha}). Note that

\begin{sub}
\item[$(\mathrm{a})$] If $a \in k^2$, then $K$ is $k$-rational.

When $\sqrt{a} \in k$, define $u=z+\sqrt{a}y$ and $v=z-\sqrt{a}y$.
(\ref{cha}) becomes $uv=P(x)$; thus $K=k(x,y,z)=k(x,u,v)=k(x,u)$
since $v=\frac{P(x)}{u} \in k(x,u)$. From now on, we will assume
that $\sqrt{a} \notin k$.

\item[$(\mathrm{b})$] Obviously we may assume that $P$ contains no multiple
irreducible factor in $k[x]$.

When $\deg P=1$, (\ref{cha}) is written as $z^2=ay^2+x$ so
$K=k(x,y,z)=k(y,z)$ is $k$-rational.

When $\deg P=0$, (\ref{cha}) is written as $z^2=ay^2+b$. Then
$K=k(x,y,z)$ is $k$-rational if and only if the quadratic form
$aY^2+bX^2=Z^2$ has a non-trivial zero over $k$, i.e. the norm residue
symbol of degree two $(a,b)_{2,k}=0$.

When $\deg P=2$ and $\ch k \not= 2$, (\ref{cha}) may be written as
$z^2=ay^2+bx^2+c$. If $c \not=0$, then $K=k(x,y,z)$ is
$k$-rational if and only if $c \in k^2-ak^2-bk^2$. If $c=0$, as
before, $K$ is $k$-rational if and only if $(a,b)_{2,k}=0$.
See Theorem 6.7 of \cite{HKO94} for details.

\item[$(\mathrm{c})$] Let $l$ be the splitting field of $P(x)$. If $\deg P \geq 3$
and $l \cap k(\sqrt{a})=k$, then $K$ is not $k$-rational by a rationality
criterion of Manin \cite{Man67}, which will be explained
in Subsection \ref{PicY} to Subsection \ref{pfA}.

\item[$(\mathrm{d})$] Suppose that some irreducible component $P_1$ of $P$ is
of the form $P_1(x)=A(x)^2-aB(x)^2$ where $A(x),B(x) \in k[x]$.
Define $z=A(x)z^\prime+aB(x)y^\prime$ and
$y=B(x)z^\prime+A(x)y^\prime$. We have
$z^2-ay^2=P_1(x)(z^{\prime^2}-ay^{\prime2})$. It follows that
$z^{\prime2}-ay^{\prime2}=P(x)/P_1(x)$.
 Since $K=k(x,y,z)=k(x,y^\prime,z^\prime)$,
 the rationality of $k(x,y,z)$ does not change if we replace $P$ by $P/P_1$.
\end{sub}

From the above discussion, we may assume the following conditions
without loss of generality.

\begin{sub}
\item[$(\mathrm{C1})$] $a \not\in k^2$.
\item[$(\mathrm{C2})$] $\deg P \geq 3$ and $P \in
k[x]$ is square-free.
\item[$(\mathrm{C3})$] If $l$ is the splitting field of
$P(x)$, then $k(\sqrt{a}) \subset l$.
\item[$(\mathrm{C4})$] Every irreducible
factor of $P(x)$ is also irreducible over $k(\sqrt{a})$, which is
equivalent to that no irreducible factor of $P(x)$ in $k[x]$ is of
the form $A(x)^2-aB(x)^2$.
\item[$(\mathrm{C5})$] $\ch k \not=2$, and every
irreducible factor of $P(x)$ is separable over $k$; this is the
assumption prescribed at the beginning of this paper.
\end{sub}

Our main result is:
\begin{thm}
The field $K=k(x,y,z)$ defined by $(\ref{cha})$ is not $k$-rational under the assumptions
$(\mathrm{C1}), \dots, (\mathrm{C5})$.
\end{thm}

\begin{rem}
The non-rationality of $K$ for the case where $\deg P=3$ (and for some other cases)
is proved by V. A. Iskovskikh and B. E. Kunyavskii (see \cite[Theorem 4.1]{Kan07}).
The case where $\deg P=3$ and $P(x)$ is irreducible is a typical example of a surface
which is not $k$-rational but stably $k$-rational
(see Beauville, Colliot-Th\'el\`ene, Sansuc and Swinnerton-Dyer \cite{BCSS85}).
\end{rem}

\bigskip
\subsection{Main result of conic bundles} \label{maincha}
Now we shall deal with the rationality of a general conic bundle,
whose function field is $K:=k(x,y,z)$ satisfying
\begin{equation} \label{cha2}
z^2=P(x)y^2+Q(x)
\end{equation}
where $P,Q \in k[x]$ are separable polynomials
and $\deg P \ge 1$, $\deg Q \ge 1$. Remember that $k$ is a field with
$\ch k \not=2$.

As before, the same problem was studied by Iskovskikh
\cite{Isk67,Isk70,Isk72} as the rationality of standard conic
bundles. Our approach is essentially an adaptation of Iskovskikh's
idea, but we will give a necessary and sufficient condition for
the rationality in terms of $P$ and $Q$ explicitly as below.

Let $s=s_1+s_2+s_3+s_4$, where $s_1$ (resp. $s_2$, resp. $s_3$) is
the number of $c \in \overline{k}$ such that $P(c)=0$ and $Q(c)
\not\in k(c)^2$ (resp. $Q(c)=0$ and $P(c) \not\in k(c)^2$, resp.
$P(c)=Q(c)=0$ and $-\frac{Q}{P}(c) \not\in k(c)^2$). $s_4=0$ or
$1$ and $s_4=1$ if and only if one of the following three
conditions are satisfied:
\begin{sub}
\item[$(\mathrm{i})$] $\deg P$ is even, $\deg Q$ odd and $p_0 \not\in k^2$;
\item[$(\mathrm{ii})$] $\deg P$ is odd, $\deg Q$ even and $q_0 \not\in k^2$;
\item[$(\mathrm{iii})$] $\deg P$ is odd, $\deg Q$ odd and $-q_0/p_0 \not\in k^2$.
\end{sub}
Here $p_0$ (resp. $q_0$) is the coefficient of the highest degree
term of $P$ (resp. $Q$).

Our main result is:
\begin{thm} \label{mainthm}
Let $K=k(x,y,z)$ be the field defined by $(\ref{cha2})$.
\hfill\break
$(1)$ When $s \geq 4$, $k(x,y,z)$ is not $k$-rational.
\hfill\break
$(2)$ When $s=3$, $k(x,y,z)$
is $k$-rational.
\hfill\break
$(3)$ The case $s=1$ can not happen.
\hfill\break
$(4)$
When $s=0$ or $s=2$, $k(x,y,z)$ is not $k$-rational
if and only if $(\mathrm{I})$ both of $\deg P$ and $\deg Q$ are even and,
$(\mathrm{II})$ $a^2p_0+b^2q_0=c^2$ has
no non-zero solution $(a,b,c)$ in $k$ (resp. $k(\sqrt{\pi_1})$) for $s=0$
(resp. $s=2$).
When $s=2$, $\pi_1$ satisfies one of the following three conditions:
$(\mathrm{i})$ $P(c_1)=0$ and $Q(c_1)=\pi_1 \not\in k^2$;
$(\mathrm{ii})$ $Q(c_1)=0$ and $P(c_1)=\pi_1 \not\in k^2$;
$(\mathrm{iii})$ $P(c_1)=Q(c_1)=0$ and $-\frac{Q}{P}(c_1)=\pi_1 \not\in k^2$.
\end{thm}

\begin{rem}
As for the case $s=2$ in Theorem \ref{mainthm} $(4)$,
if $c_1, c_2 \in k$ then $k(x,y,z)$ is $k$-rational,
otherwise $c_1$ and $c_2$ are $k$-conjugate
and we can show that $\pi_1 \in k$
(so $k_1=k(\sqrt{\pi_1})$ is a quadratic extension of $k$).
\end{rem}

\begin{rem}
The non-rationality for the case where $P(x)=x$
and $Q(x)=f(x^2)$ is discussed by
B. E. Kunyavskii, A. N. Skorobogatov and M. A. Tsfasman
(see \cite[Chapter 6]{KST89}).
\end{rem}

\subsection{Ideas of the proof} \label{idea}
The field $\overline{k}(x,y,z)$ is $\overline{k}(x)$-rational, i.e.
$\overline{k}(x,y,z)=\overline{k}(x,u)$ for some $u \in
\overline{k}(x,y,z)$. The action of $\mathfrak{G}=\Gal(k^{\sep}/k)$
on $u$ induces birational transformation of $\mathbb{P}^1 \times
\mathbb{P}^1$. After finite steps of blowings-up and down of
$\mathbb{P}^1 \times \mathbb{P}^1$, these birational
transformations become biregular on a surface $X$.
Then, the group action of $\mathfrak{G}$ induces
a permutation of irreducible curves. Thus the divisor group $\Div(X)$
becomes a permutation $\mathfrak{G}$-module,
i.e. $\Div(X)$ has a $\mathbb{Z}$-basis permuted by $\mathfrak{G}$.
Since the principal
divisor group is stable under the action of $\mathfrak{G}$, the Picard group $\Pic(X)$ is also
a $\mathfrak{G}$-module.

From the structure of $\Pic(X)$ as a $\mathfrak{G}$-module, we will
derive the $k$-irrationality of $K$. Three criteria will be
instrumental in our proof. We list them in the following.

\bigskip
I. Non-triviality of $H^1(\mathfrak{G},\Pic(X))$.

The first Galois cohomology $H^1(\mathfrak{G},\Pic(X))$ is
$k$-birational invariant (see
 \cite[pages 150-151, Theorem 2.2, Corollary 2.3]{Man69}). In particular, if $K$ is $k$-rational,
then $H^1(\mathfrak{G},\Pic(X))=0$.

The following theorem for a generalized Ch\^atelet surface
is due to Sansuc
\cite[Proposition 1 (v)]{San81}
(see also \cite{CTS94}, \cite[Proposition 7.1.1]{Sko01}). For the
convenience of the reader, we will give a proof of it in Subsection \ref{pfA}.

\begin{thm}[Sansuc \cite{San81}] \label{A}
Let $r^\prime$ be the number of irreducible components of $P(x)$ over $k$.
Define $j$ by $j=r^\prime-1$ if $\deg P$ is odd; define
$j=r^\prime-1$, if $\deg P$ is even and every irreducible
component of $P$ is of even degree; define $j=r^\prime-2$, if
$\deg P$ is even and some irreducible component of $P$ is of odd
degree. Then $H^1(\mathfrak{G},\Pic(X))=(\mathbb{Z}/2\mathbb{Z})^j$.
\end{thm}

Theorem \ref{A} implies that $K$ is not $k$-rational except when
$P(x)$ is irreducible, or a product of two irreducible polynomials
of odd degree.

\bigskip
II. Calculating the intersection form.

If $X$ is birational to $\mathbb{P}^1 \times \mathbb{P}^1$ over
$k$, there exist two families of $\mathfrak{G}$-invariant
irreducible curves $\{C_a\}$ and $\{C_a^\prime\}$ on $X$,
parametrized by elements of $k$.

After successive blowings-up at fundamental points of
$\mathbb{P}^1 \times \mathbb{P}^1$ and $X$ respectively, we will
obtain surfaces $Z$ and $Z^\prime$ which are biregular over $k$.
Except finite number of elements of $k$, $C_a$ and $C_a^\prime$
(denoted by $C$ for simplicity) satisfy the conditions that
$C \cdot C=0$ and $C \cdot \Omega=-2$ on $Z^\prime$, where
$\Omega$ is the canonical divisor.

By a blowing-up $E_j$, $C \cdot C$ decreases by $(C \cdot
E_j)^2$ and $C \cdot \Omega$ increases by $C \cdot E_j$,
so we must have
\begin{equation} \label{1_2}
C \cdot C=\sum_j m_j^2, \quad
C \cdot \Omega=-2-\sum_j m_j
\end{equation}
on $X$, where $m_j=C \cdot E_j$.

On the other hand, we will prove
\begin{thm} \label{B}
\hfill\break
$(1)$
If $K=k(x,y,z)$ is a generalized Ch\^atelet surface defined by $(\ref{cha})$
and $\deg P \geq 7$, then $K$ is not $k$-rational.
\hfill\break
$(2)$
If $K=k(x,y,z)$ is a general conic bundle defined by $(\ref{cha2})$
and $s \geq 8$ ($s$ is as in the last paragraph before Theorem \ref{mainthm}),
then $K$ is not $k$-rational.

In fact, there
is a non-singular projective surface $Y$ which is
birational to $X$, such that any
$\mathfrak{G}$-invariant irreducible curve $C$ other than
$x=$const. will not satisfy $(\ref{1_2})$
for any further blowing-up
$\{ E_j \}$.
\end{thm}

\bigskip
III. Reduction to a del Pezzo surface.

A del Pezzo surface $S$ is biregular to  some successive
blowings-up of the projective plane $\mathbb{P}^2$.

\begin{thm} \label{C}
\hfill\break
$(1)$
If $K=k(x,y,z)$ is a generalized Ch\^atelet surface defined by $(\ref{cha})$
and $3 \leq \deg P \leq 6$, then $K$ is not $k$-rational.
\hfill\break
$(2)$
If $K=k(x,y,z)$ is a general conic bundle defined by $(\ref{cha2})$
and $4 \leq s \leq 7$ ($s$ is as in the last paragraph before
Theorem \ref{mainthm}),
then $K$ is not $k$-rational.
\end{thm}

In fact, suppose $K$ is $k$-rational,
then there is a del Pezzo surface
$X^{\prime\prime}$ which is birational to $X$. Thus,
$X^{\prime\prime}$ is biregular to some successive blowings-up of
$\mathbb{P}^2.$
From this we can deduce a contradiction.

The proof of rationality when $s \leq 3$ for a general conic bundle
also uses the intersection form.
A crucial fact is that if an irreducible curve
$\Gamma$ satisfies $\Gamma \cdot \Gamma <0$, then $\Gamma$ is
unique in its class.
If it's class is $\mathfrak{G}$-invariant,
$\Gamma$ itself must be $\mathfrak{G}$-invariant.
We can find a $\mathfrak{G}$-invariant
transcendent basis of $\overline{k}(x,u)$ by using the regular
mapping from $X$ to $\mathbb{P}^2$ or $\mathbb{P}^1 \times \mathbb{P}^1$
induced by such a $\Gamma$.

Section \ref{ag} is devoted to preliminary discussions from algebraic geometry,
Section \ref{chCh} to a generalized Ch\^atelet surface and Section \ref{chCo} to a general conic bundle.


\section{Preliminaries from algebraic geometry.} \label{ag}

In this section, we shall state some results in algebraic geometry
without proof. For more details, see for instance Hartshorne
\cite{Har77}, especially Chapter 5 there.

Throughout this section, the ground field $k$ of an algebraic
variety is assumed to be algebraically closed.

\subsection{Birational mapping} \label{birational}

Let $X$ and $X^\prime$ be projective non-singular surfaces,
which are mutually birational by $T: X \rightarrow X^\prime$.
$T$ can not be defined for finite number of points
(which are called fundamental points of $T$)
because the both of numerator and denominator of $T$ becomes zero.
$T$ is not injective on finite number of irreducible curves
(which are called exceptional curves of $T$),
and $T$ maps every irreducible branch of exceptional curves to a point of $X^\prime$,
which is a fundamental point of $T^{-1}$.

The complement $O(X)$ of all fundamental points and all exceptional curves
is a Zariski open set of $X$,
and $T$ maps $O(X)$ biregularly to $O(X^\prime)$ (defined similarly for $T^{-1}$).

\begin{thm}[{\cite[Chapter 5]{Har77}}]

Every birational mapping $T: X \to X^\prime$ becomes biregular
after finite steps of blowing-up at fundamental points of $T$ and
$T^{-1}$ respectively (this is valid only for surfaces, and it
is not true for higher dimensional varieties).
\end{thm}

A concrete example of such blowings-up is given in the discussion in Subsection \ref{biregularization}.

\subsection{Blowing-up}

Let $X$ be a projective non-singular surface, and $P$ be a point on $X$.
Then, there exists uniquely (modulo biregularity) a projective non-singular surface
$\widetilde X$ which satisfies the following $(1), (2)$ and $(3)$.
$\widetilde X$ is called the blowing-up of $X$ at $P$.
\begin{sub}
\item[$(1)$]
$X$ and $\widetilde X$ are mutually birational by $\pi: \widetilde
X \to X$.
\item[$(2)$]
$\pi$ is regular and has no fundamental point.
$\pi$ has a unique exceptional curve $E_p$,
which is biregular to the projective line $\mathbb{P}^1$,
and $\pi$ maps $E_P$ to $P$.
\item[$(3)$]
$\pi^{-1}$ has a unique fundamental point $P$
and has no exceptional curve.
\end{sub}
In other words, $X \setminus \{P\}$ and $\widetilde X \setminus
E_P$ are mapped biregularly and $\pi$ maps $E_P$ to $P$ while
$\pi^{-1}$ is not defined at $P$.

Roughly speaking, $\widetilde X$ is the dilation of a point $P$ to a line $E_P$ in $\widetilde X$.
In the tangent plane of $X$ at $P$,
the direction ratios of tangent vectors correspond to
points on $E_P$.
Thus $E_P$ is the set of direction ratios of tangent vectors at $P$.

\subsection{$\Div(X)$ and $\Pic(X)$.}

Let $X$ be a projective non-singular surface.
The divisor group $\Div(X)$ is defined as the free $\mathbb{Z}$-module
with all irreducible curves on $X$ as basis.
Every irreducible curve $C$ on $X$ induces a valuation $v_C$
on the function field $k(X)$,
and for $f \in k(X)$, the divisor $\sum v_C(f)C$
is called the principal divisor of $f$.

When $f$ runs over $k(X)$,
the principal divisors form a subgroup of $\Div(X)$,
which are called the principal divisor group.
It is isomorphic to $k(X)^\times/k^\times$.

The factor group of $\Div(X)$ by the principal divisor group is called
the divisor class group or Picard group and denoted by $\Pic(X)$.

\begin{rem}
In more general setting,
the definition of Picard group is more complicated,
but for a projective non-singular surface,
it is nothing but the divisor class group.
\end{rem}

Let $\widetilde X$ be the blowing-up of $X$ at $P$.
Let $C$ be an irreducible curve on $X$.
If $C$ does not pass through $P$,
then $\widetilde C:=\pi^{-1}(C)$
is an irreducible curve on $\widetilde X$.
If $C$ passes through $P$,
let $\widetilde C$ be the Zariski closure of
$\pi^{-1}(C \setminus \{P\})$
in $\widetilde X$,
then $\widetilde C$ is an irreducible curve on $\widetilde X$.
Besides $\widetilde C$, the only one irreducible curve
on $\widetilde X$ is $E_P$.
So identifying $C$ and $\widetilde C$,
we have $\Div(\widetilde X)=\Div(X) \oplus \mathbb{Z}$,
where $\mathbb{Z}$ represents the free $\mathbb{Z}$-module
with $E_P$ as the base.

Since $X$ and $\widetilde X$ are birational,
the function fields are the same,
$k(\widetilde X)=k(X)$.
Taking the factor group by the common principal divisor group,
we have $\Pic(\widetilde X) \simeq \Pic(X) \oplus \mathbb{Z}$,
where $\mathbb{Z}$ represents the free $\mathbb{Z}$-module
with $E_P$ as the base.

We shall give the isomorphism more explicitly in the next subsection,
using the intersection forms.

\subsection{Intersection form.} \label{intersection}

\begin{thm}[{{\cite[Chapter 5]{Har77}}}] \label{thmintform}

On $\Div(X) \times \Div(X)$,
there exists uniquely a symmetric $\mathbb{Z}$-bilinear form
$D_1 \cdot D_2$ satisfying the following conditions.
It is called the intersection form.
\hfill\break
$(1)$
If two irreducible curves $C_1$ and $C_2$ do not intersect on $X$,
then $C_1 \cdot C_2 = 0$.
\hfill\break
$(2)$
If $C_1$ and $C_2$ intersects transversally at $n$ points,
then $C_1 \cdot C_2=n$.
Here ``intersects transversally at $P$''
means that both $C_1$ and $C_2$ are non-singular at $P$,
and tangent vectors of $C_1$ and $C_2$ at $P$ are linearly independent.
\hfill\break
$(3)$
If $D$ is a principal divisor,
then $D \cdot D^\prime=0$ for all $D^\prime \in \Div(X)$.
So that the intersection form is defined on $\Pic(X) \times \Pic(X)$.
\end{thm}
If $C_1$ and $C_2$ intersect at $n$ points,
but not transversally at some point,
then we have $C_1 \cdot C_2 >n$.
So, for every two different irreducible curves $C_1$, $C_2$,
we have $C_1 \cdot C_2 \geq 0$.
But $C \cdot C$ (called the self-intersection number of $C$)
can be $<0$.
Note that $C \cdot C$ is determined indirectly using the condition (3).

The relation of the intersection form and blowing-up is as follows.

First, consider $E_P \cdot \widetilde C$.
From (1) and (2) above, we have
\begin{sub}
\item[$(1^\prime)$]
If $C$ does not pass through $P$,
then $E_P \cdot \widetilde C=0$.
\item[$(2^\prime)$]
If $C$ passes through $P$,
and $C$ is non-singular at $P$,
then $E_P \cdot \widetilde C=1$.
\item[$(3^\prime)$]
Suppose that $C$ passes through $P$,
and $C$ is singular at $P$.
The local equation of $C$ is given by $F(x,y)=0$
where $x$ and $y$ are local coordinates at $P$
with $x=y=0$,
and $F(x,y)$ is a formal power series of $x$ and $y$.
Since $C$ passes through $P$,
the constant term of $F$ is zero.
Since $C$ is singular at $P$,
the coefficients of $x$ and $y$ are also zero.
Let $\nu$ be the smallest integer of $i+j$
such that the coefficient of $x^iy^j$ is not zero,
then $E_P \cdot \widetilde C=\nu$.
\end{sub}

Note that the homogeneous part of degree $\nu$ of $F$ induces a polynomial of degree $\nu$
in $\frac{y}{x}$,
so there are $\nu$ roots of $\frac{y}{x}$.

Using this $E_P \cdot \widetilde C$, we have
\begin{equation} \label{2_1}
\widetilde C_1 \cdot \widetilde C_2=C_1 \cdot C_2 -
(E_P \cdot \widetilde C_1)(E_P \cdot \widetilde C_2).
\end{equation}
For simplicity, suppose that $\widetilde C_1$ and $\widetilde C_2$ do not intersect on $E_P$.
Since $C_1$ passes $(E_P \cdot \widetilde C_1)$ times through $P$
and $C_2$ passes $(E_P \cdot \widetilde C_2)$ times through $P$,
there are $(E_P \cdot \widetilde C_1)(E_P \cdot \widetilde C_2)$
virtual intersection points on $X$.
This verifies the formula $(\ref{2_1})$.
More considerations show that the above formula $(\ref{2_1})$ is valid
even if $\widetilde C_1$ and $\widetilde C_2$ intersect on $E_P$.
Finally we have
\begin{equation}
E_P \cdot E_P=-1,
\end{equation}
which is obtained using the condition (3) in Theorem \ref{thmintform}.

Considering the valuation $v_{{}_C}$, we see that if $D$ is a principal divisor on $X$,
then $\widetilde D+(E_P \cdot \widetilde D)E_P$ is a principal divisor on $\widetilde X$.
This derives the following fact.

Let $\pi^\ast$ be a $\mathbb{Z}$-linear map $\Div(X) \rightarrow \Div(\widetilde X)$
defined by $\pi^\ast(D)=\widetilde D+(E_P \cdot \widetilde D)E_P$.
Then $\pi^\ast$ is injective and maps the principal divisor group
to the principal divisor group.
So taking the factor group, we get the isomorphism
$\Pic(\widetilde X) \simeq \Pic(X) \oplus \mathbb{Z}$.

Even if $D_1 \equiv D_2$ ($\equiv$ means the identity modulo principal divisor group),
\begin{equation}
\widetilde D_1 \equiv \widetilde D_2 +
\big\{(E_P \cdot \widetilde D_2)-(E_P \cdot \widetilde D_1)\big\}E_P.
\end{equation}

\subsection{Canonical divisor}

Let $X$ be a projective non-singular surface.
A canonical divisor of $X$ is defined as follows.

Let $f,g \in k(X)$ be mutually algebraic independent.
Let $C$ be an irreducible curve on $X$ and $P$ be a non-singular point of $C$.
Take a local coordinate $(x,y)$ at $P$
and consider the Jacobian
$\frac{\partial (f,g)}{\partial (x,y)}=
\begin{vmatrix}
\frac{\partial f}{\partial x} & \frac{\partial f}{\partial y} \\
\frac{\partial g}{\partial x} & \frac{\partial g}{\partial y} \\
\end{vmatrix}$,
then we can show that
$v_{{}_C}\Big(\frac{\partial (f,g)}{\partial (x,y)}\Big)$
is independent of the choice of a point $P$
and the choice of a local coordinate $(x,y)$.
Canonical divisor of $(f,g)$
is defined as $\sum v_{{}_C}\Big(\frac{\partial (f,g)}{\partial (x,y)}\Big)C$.

Take another $f_1,g_1 \in k(X)$ mutually algebraic independent.
Then canonical divisor of $(f_1,g_1)$ belongs to the same divisor class
with that of $(f,g)$,
namely all canonical divisors determine the unique divisor class in $\Pic(X)$.
This is called the canonical divisor class of $X$ and denoted by $\Omega$.

\begin{rem}
In more general setting, the definition of the canonical divisor class is more complicated,
but for a projective non-singular surface $X$,
it is nothing but the one defined above.
\end{rem}

\begin{exmp}
For $\mathbb{P}^1 \times \mathbb{P}^1$,
we shall determine the intersection form and the canonical divisor.

$\Pic(\mathbb{P}^1 \times \mathbb{P}^1)$ has rank 2 as a $\mathbb{Z}$-module
with the basis $x=\infty$ and $u=\infty$
(irreducible curves which do not come from irreducible polynomials in
$\overline{k}[x,u]$ are $x=\infty$ and $u=\infty$).
The class of an irreducible curve $C$ of the degree $n$ with respect to $x$ and $m$
with respect to $u$ is $nF+mF^\prime$
where $F$ is the class of $(x=\infty)$ and $F^\prime$ is the class of $(u=\infty)$.
For any $c, c^\prime \in \overline{k}$,
the representatives of $F$ and $F^\prime$ are chosen as $x=c$ and $u=c^\prime$ respectively.

The intersection form on $\mathbb{P}^1 \times \mathbb{P}^1$
is determined by
\begin{equation}
F \cdot F = F^\prime \cdot F^\prime = 0, \, F \cdot F^\prime =1.
\end{equation}
The canonical divisor is
\begin{equation} \label{ex}
\Omega=-2F-2F^\prime.
\end{equation}
Take $f=x$ and $g=u$,
then we have $\frac{\partial (x,u)}{\partial (x,u)}=1$
since $(x,u)$ is a local coordinate except on the lines
$(x=\infty)$ and $(u=\infty)$.
In a neighborhood of the line $(x=\infty)$,
a local coordinate is $(t,u)$
where $t=\frac{1}{x}$,
so $x=\frac{1}{t}$,
then $\frac{\partial (x,u)}{\partial (t,u)}=-\frac{1}{t^2}$,
thus $v_{(x=\infty)}\Big(\frac{\partial (x,u)}{\partial (t,u)}\Big)=-2$.
The similar result holds for the line $(u=\infty)$.
This verifies $(\ref{ex})$.
From $(\ref{ex})$ we see that
\begin{equation}
C \cdot \Omega =-2(m+n), \,
\Omega \cdot \Omega=8.
\end{equation}
\end{exmp}

Return to a general $X$ and we shall consider the relation with the blowing-up.
Let $\widetilde X$ be the blowing-up of $X$ at a point $P$.
Then the canonical divisor of $\widetilde X$ is given by
\begin{equation} \label{OmegaX}
\Omega_{\widetilde X}
=\pi^\ast \Omega_X+E_P.
\end{equation}
This can be derived as follows.
Let $f=x$ and $g=y$,
where $(x,y)$ is a local coordinate of $X$ at $P$
with $x=y=0$ at $P$.
Since $\frac{\partial (f,g)}{\partial (x,y)}=1$,
$\Omega$ does not pass through $P$,
so $\widetilde \Omega \cdot E_p=0$.
On the other hand,
a local coordinate of $\widetilde X$ in a neighborhood of $E_p$
is $(x,t)$ where $t=\frac{y}{x}$,
so $y=tx$, then $\frac{\partial(x,y)}{\partial(x,t)}=x$,
thus $v_{E_P}\big(\frac{\partial(x,y)}{\partial(x,t)}\big)=1$.
This implies $\Omega_{\widetilde{X}}=\widetilde{\Omega}_X+E_P$.

For other canonical divisors,
extending the above relation in the form compatible with
the action of $\pi^\ast$,
we get $(\ref{OmegaX})$ above.

Since $\pi^\ast C_1 \cdot \pi^\ast C_2 =C_1 \cdot C_2$
and $\pi^\ast C \cdot E_P=0$ for any irreducible curve
$C, C_1, C_2$ on $X$,
from $(\ref{OmegaX})$ we have
\begin{eqnarray} \label{widetildeC}
\widetilde{C} \cdot \Omega_{\widetilde X} = C \cdot \Omega_X
+ \widetilde C \cdot E_P, \\
E_P \cdot \Omega_{\widetilde X}=-1, \,
\Omega_{\widetilde X} \cdot \Omega_{\widetilde X}=
\Omega_X \cdot \Omega_X-1. \nonumber
\end{eqnarray}

\subsection{Blowing down}

Blowing-down is the inverse operation of the blowing up.
Let $X$ be a projective non-singular surface,
and assume that there exists an irreducible curve $L$ on $X$
satisfying $L \cdot L=-1$
and $\Omega \cdot L=-1$
($L$ is necessarily biregular to the projective line $\mathbb{P}^1$).

\begin{thm}
There exists a unique (modulo biregularity)
projective non-singular surface $\overline{X}$
such that the blowing-up $\widetilde{\overline{X}}$
at some point $Q \in \overline{X}$
is biregular to $X$,
mapping $E_Q$ to $L$.
\end{thm}

The surface $\overline{X}$ is called the blowing-down of $X$ by $L$.
Let $\varphi$ be the biregular mapping $X \rightarrow
\widetilde{\overline{X}}$,
and $\pi$ be the projection
$\widetilde{\overline{X}} \mapsto \overline{X}$.
For an irreducible curve $C \not= L$ on $X$,
let $\overline{C}$ be the image of $C$ by $\pi \circ \varphi$.
Then $\overline{C}$ is an irreducible curve of $\overline{X}$
and all irreducible curves on $\overline{X}$
are obtained in this way.
So that identifying $C$ with $\overline{C}$,
we get $\Div(X)=\Div(\overline{X}) \oplus \mathbb{Z}$,
where $\mathbb{Z}$ represents the free $\mathbb{Z}$-module
with $L$ as the basis.

Let $\overline{\pi}$ be the $\mathbb{Z}$-linear map from $\Div(X)$ to $\Div(\overline{X})$
defined by $\overline{\pi}(D)=\overline{D-\lambda L}$,
where $\lambda$ is the coefficient of $L$ in $D$.
Then $\overline{\pi}$ is surjective and
maps the principal divisor group
to the principal divisor group bijectively.
The kernel of $\overline{\pi}$
is the free $\mathbb{Z}$-module with $L$ as the basis.
So $\overline{\pi}$ induces the isomorphism
$\Pic(X) \simeq \Pic(\widetilde X) \oplus \mathbb{Z}$.

The intersection form on $\overline{X}$ is given by
\begin{equation}
\overline{D}_1 \cdot \overline{D}_2=
D_1 \cdot D_2 + (D_1 \cdot L)(D_2 \cdot L).
\end{equation}
The canonical divisor of $\overline{X}$ is given by
\begin{equation}
\Omega_{\overline{X}}=
\overline{\pi}(\Omega_X)=\overline{\Omega_X-\lambda L}.
\end{equation}
We have
\begin{eqnarray}
\overline{D} \cdot \Omega_{\overline{X}} &=& D \cdot \Omega_X - D \cdot L, \\
\Omega_{\overline{X}} \cdot \Omega_{\overline{X}} &=&
\Omega_X \cdot \Omega_X+1. \nonumber
\end{eqnarray}

\subsection{Blowing-up and down.} \label{updown}

Let $X$ be a projective non-singular surface and $F$ be an irreducible curve
on $X$ satisfying $F \cdot F=0$ and $F \cdot \Omega=-2$
($F$ is necessarily biregular to the projective line $\mathbb{P}^1$).
Consider the blowing-up $\widetilde X$ at a point $P$ on $F$.
Then we have $\widetilde F \cdot \widetilde F=-1$
and $\widetilde F \cdot \Omega_{\widetilde X}=-1$,
so that we can consider the blowing-down of $\widetilde X$ by $\widetilde F$
and obtain $\overline{\widetilde{X}}$.

$X$ and $\overline{\widetilde X}$ are birational,
but not regular in any direction.
Let $\pi_1$ be the projection $\widetilde X \rightarrow X$
and $\pi_2$ be the projection $\widetilde{\overline{\widetilde{X}}}
\rightarrow \overline{\widetilde{X}}$,
then $\rho=\pi_2 \circ \varphi \circ \pi_1^{-1}$
is the birational mapping from $X$ to $\overline{\widetilde{X}}$.

The fundamental point of $\rho$ is $P$,
and the exceptional curve of $\rho$ is $F$.
On the other hand,
the fundamental point of $\rho^{-1}$ is $Q$,
and the exceptional curve of $\rho^{-1}$ is $\overline{E}_P$
($Q$ is a point on $\overline{E}_P$,
because $\widetilde{\overline{E}}_P \cdot E_Q=
E_P \cdot \widetilde F=1$).

For an irreducible curve $C \not=F$ on $X$,
$\overline{\widetilde{C}}$ is an irreducible curve on $\overline{\widetilde{X}}$,
and besides them,
$\overline{E}_P$ is the only irreducible curve on $\overline{\widetilde{X}}$.
So that $\Div(X) \simeq \Div(\overline{\widetilde{X}})$,
but $F$ is omitted from the basis of $\Div(X)$
and $\overline{E}_P$ is added as the basis of $\Div(\overline{\widetilde{X}})$.

However, we need not replace the basis for $\Pic$.
Let $\rho^\ast=\overline{\pi}_2 \circ \pi_1^\ast$ be
the $\mathbb{Z}$-linear map from $\Div(X)$ to $\Div(\overline{\widetilde{X}})$.
The map $\rho^\ast$ is written as
\begin{equation}
\rho^\ast(D)=\overline{\widetilde{D-\lambda F}}+
(\widetilde D \cdot E_P) \overline{E}_P.
\end{equation}
The map $\rho^\ast$ maps $\Div(X)$ to $\Div(\overline{\widetilde{X}})$ bijectively,
and maps the principal divisor group to the principal divisor group.
So, $\rho^\ast$ induces an isomorphism of $\Pic(X)$ to $\Pic(\overline{\widetilde X})$.
Since $\rho^\ast$ maps $F$ to $\overline{E}_P$,
the divisor class of $F$ is mapped to the divisor class of $\overline{E}_P$.
(More precisely,
for a divisor $D$ on $X$,
$D \equiv F$ on $X$ is equivalent with
$\rho^\ast(D) \equiv \overline{E}_P$.)

The intersection form on $\overline{\widetilde{X}}$ is given as follows.
\begin{eqnarray}
\overline{E}_P \cdot \overline{E}_P=0,\
\overline{\widetilde{C}} \cdot \overline{E}_P=C \cdot F
\text{ for } C \not= F, \nonumber \\
\overline{\widetilde{C}}_1 \cdot \overline{\widetilde{C}}_2=
C_1 \cdot C_2+(C_1 \cdot F)(C_2 \cdot F)
-(C_1 \cdot F)(\widetilde C_2 \cdot E_P) \\
-(\widetilde C_1 \cdot E_P)(C_2 \cdot F). \nonumber
\end{eqnarray}
The canonical divisor of $\overline{\widetilde{X}}$ is given by
\begin{equation}
\Omega_{\overline{\widetilde{X}}}=
\rho^\ast(\Omega_X)+\overline{E}_P=
\overline{\widetilde{\Omega_X-\lambda F}}+
\big\{ (\widetilde{\Omega}_X \cdot E_P)+1\big\}
\overline{E}_P.
\end{equation}
Of course we have
$\Omega_{\overline{\widetilde{X}}} \cdot
\Omega_{\overline{\widetilde{X}}} =
\Omega_X \cdot \Omega_X$
and $\overline{E}_P \cdot \Omega_{\overline{\widetilde X}}=-2$.

\subsection{Iteration of blowings-up and down.}

Let $X$ be a non-singular projective surface,
and $P_1, P_2, \dots, P_r$ be points on $X$.
The successive blowings-up at $\{P_i\}_{1 \leq i \leq r}$
does not depend on the order of the blowings-up
(more precisely, the obtained surface by the blowings-up
in different orders are mutually biregular).

For successive blowings-up on the once blowing-up $E$, (namely
$E_1$ is the blowing-up at $P_1 \in X$, $E_2$ is the blowing-up at
$P_2 \in E_1$, $E_3$ is the blowing-up at $P_3 \in E_2$, and so
on) the order of the blowing-up can not be changed. In this case
$C \cdot E_i$ is monotonically decreasing.

Let $X_1$ be the blowing-up of $X$ at a point $P_1$,
and $Y_1$ be the blowing-down by some $\widetilde{F}_1$,
where $F_1$ is an irreducible curve on $X$
passing through $P_1$
such that $F_1 \cdot F_1=0$ and $F_1 \cdot \Omega=-2$ on $X$.

The blowing-up of $Y_1$ at some point $Q_1$ of $Y_1$ is
biregular with $X_1$, mapping $E_Q$ to $\widetilde{F}_1$,
by the definition of the blowing-down.

Let $X_2$ be the blowing-up of $X_1$ at $P_2 \in X_1 \setminus \widetilde{F}_1$.
Since $X_1 \setminus \widetilde{F}_1$ is biregular
with $Y_1 \setminus \{Q_1\}$,
this induces a blowing-up of $Y_1$ at the corresponding point $P_2^\prime$.
Blow-down again by some $\widetilde{F}_2$,
and let $Y_2$ be the obtained surface.

The blowing-up of $Y_2$ at some point $Q_2$ is biregular
with the blowing-up of $Y_1$ at $P_2^\prime$.
Since $X_2$ is biregular with the successive blowings-up of $Y_1$
at $Q_1$ and $P_2^\prime$,
we see that $X_2$ is biregular with the successive blowings-up of $Y_2$
at $Q_2$ and $Q_1$.

Repeat this $r$-times.
Let $X_r$ be the surface obtained from $X$ by the successive blowings-up at $\{P_i\}$.
After each blowing-up, take a suitable blowing-down,
and after repeating this $r$-times,
let $Y_r$ be the obtained surface.
Then $X_r$ is biregular with the successive blowings-up of $Y_r$ at $\{Q_i\}$.
Here we assume that $P_i$ does not lie on $F_j$ for $j<i$ in $X_{i-1}$
(for simplicity, we omit $\widetilde{}$ and $\overline{\ }$
for blowing-up and down).

\subsection{The surface $Y_{rs}$} \label{pre2}

Let $X_1$ be the blowing up of $\mathbb{P}^1 \times \mathbb{P}^1$ at $(a,b)$.
$\Pic(X_1)$ has rank 3 with basis $F, F^\prime$ and $E_1$,
where $E_1$ is the blowing up of the base point $(a,b)$.
The intersection form is the same as $\Pic(\mathbb{P}^1 \times \mathbb{P}^1)$
for $F$ and $F^\prime$ and $E_1 \cdot F=E_1 \cdot F^\prime=0,
E_1 \cdot E_1=-1$
(we take the representative of $F$ as $x=c \not= a$
and the representative of$F^\prime$ as $u=c^\prime \not= b$).
The class of an irreducible curve $C$ of the degree $n$ with respect to $x$
and $m$ with respect to $u$ is $nF+mF^\prime-m_1E_1$
where $m_1=C \cdot E_1$.
The canonical divisor is $\Omega=-2F-2F^\prime+E_1$,
so $\Omega \cdot \Omega=7$.

Let $Y$ be the blowing down of $X_1$ by $x=a$.
$\Pic(Y)$ has rank 2 with the basis $F$ and $F^\prime$.
But the intersection form is different from that of
$\Pic(\mathbb{P}^1 \times \mathbb{P}^1)$ and
$F \cdot F=0, \ F \cdot F^\prime=F^\prime \cdot F^\prime=1$.
The class of the above mentioned $C$ is $(n-m_1)F+mF^\prime$.
In addition to $x=c \, (c \not=a)$,
$E_1$ also belongs to the class $F$.
The canonical divisor is $\Omega=-F-2F^\prime$,
so $\Omega \cdot \Omega=8$.
For simplicity, we omit $\widetilde{\,}$ and $\overline{\mbox{ }}$
for blowing-up and down.
The confusion is avoided by seeing $C$ is a curve on which surface.

Starting from $Y$, consider a similar blowing up and down
and let $Y_2$ be the obtained surface.
Repeat this procedure and let $Y_r$ be the surface obtained
by $r$-times blowing up and down.
Let $Y_{rs}$ be an $s$-point blow up of $Y_r$.
$\Pic(Y_{rs})$ has rank $s+2$ with the basis
$F, F^\prime$ and $E_i (1 \leq i \leq s)$.
The intersection form is $F \cdot F=0, \ F \cdot F^\prime=1, \
F^\prime \cdot F^\prime=r, \ E_i \cdot F=E_i \cdot F^\prime
=E_i \cdot E_j = 0 (i \not= j)$
and $E_i \cdot E_i=-1$.
The class of the above mentioned $C$ is
$(n-\sum_{j=1}^r m_j^\prime)F+mF^\prime-\sum_{i=1}^s m_iE_i$
with $m_i=C \cdot E_i$ and $m_j^\prime=C \cdot E_j^\prime$
where $E_j^\prime$ is the blowing up used for obtaining $Y_r$.
The canonical divisor is $\Omega=(r-2)F-2F^\prime+\sum_{i=1}^s E_i$,
so $\Omega \cdot \Omega=8-s$.

\subsection{Del Pezzo surface.}

\begin{defn}
A non-singular projective surface $X$ is called a del Pezzo
surface if it is rational (namely, birational with $\mathbb{P}^2$
or $\mathbb{P}^1 \times \mathbb{P}^1$ over $\overline{k}$) and the
anti-canonical divisor is ample. The latter condition means that
$\Omega \cdot \Omega >0$ and $\Omega \cdot \Gamma <0$ for every
irreducible curve $\Gamma$ on $X$.
The degree $\omega$ of a del Pezzo surface $X$ is
defined to be the self intersection number $\Omega \cdot \Omega$.
\end{defn}

The following is a fundamental theorem for a del Pezzo surface.

\begin{thm} \label{lem1}
A del Pezzo surface with $\omega \leq 7$ is biregular with
$(9-\omega)$-point blow up of $\mathbb{P}^2$ where $\omega=\Omega
\cdot \Omega$.
\end{thm}

Proof can be found in Nagata \cite{Nag60a,Nag60b} or Manin
\cite{Man86}.

 A del Pezzo surface with $\omega=8$ is biregular with
$\mathbb{P}^1 \times \mathbb{P}^1$ or one point blow up of
$\mathbb{P}^2$.

Conversely, a $d$-point blow up of $\mathbb{P}^2$ is a del Pezzo surface
if and only if $d \leq 8$ and
\begin{sub}
\item[$(1)$]
any 3 points do not lie on the same line $(d \geq 3)$,
\item[$(2)$]
any 6 points do not lie on the same quadratic curve $(d \geq 6)$,
\item[$(3)$]
there exists no cubic curve which passes through all 8 points and singular at one of them
$(d=8)$.
\end{sub}

\begin{thm} \label{lem2}
On a del Pezzo surface $S$,
consider the following condition
$(\ref{cond1})$ on a class $F \in \Pic(S)$:
\begin{equation} \label{cond1}
F \cdot F=0, \, F \cdot \Omega=-2 \,
\mbox{ and } F
\mbox{ contains an irreducible curve.}
\end{equation}
Then we have
\hfill\break
$(1)$
For any point $P \in S$,
there exists a unique curve $C \in F$,
which passes through $P$
($C$ is irreducible except finite number of them),
\hfill\break
$(2)$
$-\Omega-F$ (resp. $-2\Omega-F$, resp. $-4\Omega-F$)
also satisfies the condition $(\ref{cond1})$
for $\omega=4$ (resp. $\omega=2$, resp. $\omega=1$).
\end{thm}

\begin{thm} \label{lem3}
On a del Pezzo surface $S$,
consider the following condition $(\ref{cond2})$ on a class $\Gamma \in \Pic(S)$:
\begin{equation} \label{cond2}
\Gamma \cdot \Gamma=-1, \, \Gamma \cdot \Omega=-1\,
\mbox{ and } \Gamma
\mbox{ contains an irreducible curve.}
\end{equation}
Obviously the irreducible curve is unique in its class,
so denote it by the same symbol $\Gamma$.
Then $-\Omega-\Gamma$ (resp. $-2\Omega-\Gamma$)
also satisfies the condition $(\ref{cond2})$
for $\omega=2$ (resp. $\omega=1$).
\end{thm}

For a $d$-point blow-up of $\mathbb{P}^2$,
we can write down explicitly all the classes which satisfy
$(\ref{cond1})$ or $(\ref{cond2})$,
and check the validity of Theorem \ref{lem2} and Theorem \ref{lem3}.
Theorem \ref{lem2} and Theorem \ref{lem3} are valid also for a general del Pezzo surface
by Theorem \ref{lem1}.


\section{Generalized Ch\^atelet surface} \label{chCh}

We shall examine the rationality of $K=k(x,y,z)$ as in $(\ref{cha})$
$z^2=ay^2+P(x)$ under the assumptions $(\mathrm{C1}), \dots, (\mathrm{C5})$ in Subsection \ref{main}.

Let $l$ be the splitting field of $P(x)$. Then $\sqrt{a} \in l$
because of the condition $(\mathrm{C3})$. Thus $l(x,y,z)$ is $l$-rational and
$l(x,y,z)=l(x,u)$ where $u=z+\sqrt{a}y$. The field $l$ is a Galois
extension of $k$, and we write $\mathfrak{G}=\Gal(l/k)$ and
$N=\Gal\big(l/k(\sqrt{a})\big)$. The group $\mathfrak{G}$ acts on $x$
trivially, and $N$ acts on $u$ trivially; for any $\sigma \in
\mathfrak{G} \setminus N$, $\sigma :u \mapsto
z-\sqrt{a}y=\frac{P(x)}{u}$.

The automorphism $T:(x,u) \mapsto (x,\frac{P(x)}{u})$ of $l(x,u)$
induces an $\overline{l}$-birational transformation of
$\mathbb{P}^1 \times \mathbb{P}^1$. After successive blowings-up
and blowings-down, we obtain a surface $X$ defined over
$l$, on which $T$ acts as a biregular automorphism.

\subsection{Biregularization of $T$.} \label{biregularization}

Let $T$ be the birational transformation of $\mathbb{P}^1 \times \mathbb{P}^1$
defined by $T: x \mapsto x, u \mapsto \frac{P(x)}{u}$.
Let $r=\deg P$ and $c_1,c_2,\dots,c_r$ be the roots of $P$.
$T$ has $r+1$ fundamental points and $r+1$ exceptional curves.
Fundamental points are $P_i:x=c_i, u=0 \, (1 \leq i \leq r)$ and $P_{r+1}:x=u=\infty$.
Exceptional curves are $x=c_i \, (1 \leq i \leq r)$ and $x=\infty$.

Consider the blowings-up for each $P_i$.
Let $X_1$ be the blowing-up of $\mathbb{P}^1 \times \mathbb{P}^1$ at $P_1$
and $T_1$ be the lifting of $T$ to $X_1$.
$X_1$ is a surface in $\mathbb{P}^1 \times \mathbb{P}^1 \times \mathbb{P}^1$
defined by $\frac{u}{x-c_1}=t_1$.
$E_1$ is the curve $x=c_1, u=0$,
and $\widetilde{x=c_1}$ is the curve $x=c_1, t_1=\infty$.

We write $P(x)=b\prod_{j=1}^r(x-c_j)$.
By $T_1$, each point $(c_1,u,\infty) \in (\widetilde{x=c_1})$
is mapped to $(c_1,0,\frac{b\prod_{j\not=1}(c_1-c_j)}{u}) \in E_1$,
and each point $(c_1,0,t_1) \in E_1$
is mapped to $(c_1,\frac{b\prod_{j \not= 1}(c_1-c_j)}{t_1},\infty)
\in (\widetilde{x=c_1})$.
So $T_1$ maps $E_1$ biregularly to $\widetilde{x=c_1}$.

Next let $X_2$ be the blowing-up of $X_1$ at $P_2$
and $T_2$ be the lifting of $T_1$ to $X_2$.
Repeat this $r$ times
so that $T_r$ maps $E_i$ biregularly to $\widetilde{x=c_i}$
for $1 \leq i \leq r$:

$$\begin{matrix}
X_r & \stackrel{T_r}{\longrightarrow} & X_r \\
\downarrow & & \downarrow \\
\vdots & & \vdots \\
\downarrow & & \downarrow \\
X_2 & \stackrel{T_2}{\longrightarrow} & X_2 \\
\downarrow & & \downarrow \\
X_1 & \stackrel{T_1}{\longrightarrow} & X_1 \\
\downarrow & & \downarrow \\
\mathbb{P}^1 \times \mathbb{P}^1 &
\stackrel{T}{\longrightarrow} & \mathbb{P}^1 \times \mathbb{P}^1. \\
\end{matrix}$$

Now the only fundamental point of $T_r$ is $P_{r+1}: x=u=\infty$,
and the only exceptional curve of $T_r$ is $x=\infty$.

The blowing-up at $P_{r+1}$ does not make $T_{r+1}$ biregular.
Let $X_{r+1}$ be the blowing-up of $X_r$ at $P_{r+1}$
and $T_{r+1}$ be the lifting of $T_r$ to $X_{r+1}$.
$X_{r+1}$ is a surface in $X_r \times \mathbb{P}^1$
defined by $\frac{u}{x}=t_{r+1}$.
$E_{r+1}$ is the curve $x=\infty, u=\infty$
and $\widetilde{x=\infty}$ is the curve $x=\infty, t_{r+1}=0$.

$T_{r+1}$ maps $\widetilde{x=\infty}$ to
one point $P_{r+2} \in E_{r+1}$ defined by $t_{r+1}=\infty$,
and maps $E_{r+1}$ to $P_{r+2}$.
So exceptional curves of $T_{r+1}$ are
$\widetilde{x=\infty}$ and $E_{r+1}$
while the only fundamental point is $P_{r+2}$.

So, blow up again.
Let $X_{r+2}$ be the blowing-up of $X_{r+1}$ at $P_{r+2}$
and $T_{r+2}$ be the lifting of $T_{r+1}$ to $X_{r+2}$.
$X_{r+2}$ is a surface in $X_{r+1} \times \mathbb{P}^1$ defined by
$\frac{t_{r+1}}{x}=t_{r+2}=\frac{u}{x^2}$.
Then $T_{r+2}$ maps $\widetilde{x=\infty}$ to one point $P_{r+3} \in E_{r+2}$
defined by $t_{r+2}=\infty$.
If $r=3$, then $T_5$ maps $E_4$ biregularly to $E_5$,
but if $r>3$, then $T_{r+2}$ maps both of $E_{r+1}$ and $E_{r+2}$ to one point $P_{r+3}$.

Repeating this $r$ times.
Let $X$ be the obtained surface
and $T_{2r}$ be the lifting of $T$ to $X$.
Then $T_{2r}$ becomes biregular,
namely $T_{2r}$ maps $\widetilde{x=\infty}$ to $E_{2r}$,
and $E_{r+i}$ to $E_{2r-i}$ $(1 \leq i \leq r-1)$ biregularly:

$$\begin{matrix}
X & \stackrel{T_{2r}}{\longrightarrow} & X \\
\downarrow & & \downarrow \\
\vdots & & \vdots \\
\downarrow & & \downarrow \\
X_{r+2} & \stackrel{T_{r+2}}{\longrightarrow} & X_{r+2} \\
\downarrow & & \downarrow \\
X_{r+1} & \stackrel{T_{r+1}}{\longrightarrow} & X_{r+1} \\
\downarrow & & \downarrow \\
X_r & \stackrel{T_r}{\longrightarrow} & X_r. \\
\end{matrix}$$

Thus, $T$ becomes biregular after $2r$ blowings-up in total,
once for each $1 \leq i \leq r$, and $r$ times for $P_{r+1}$.
We denote the obtained surface by $X$.

\subsection{Reduction to the even degree case.} \label{degeven}

Without loss of generality,
we can assume that $\deg P=r$ is even,
by the following reason.

Suppose that $\deg P=r$ is odd, and put $r=2s-1$.
Put $x^\prime=\frac{1}{x}, y^\prime=x^{\prime s}y,
z^\prime=x^{\prime s}z$,
then $z^2=ay^2+P(x)$ is re-written as
$z^{\prime 2}=ay^{\prime 2}+x^{\prime2s}P(\frac{1}{x^\prime}).$
When $P(x)=\sum_{i=0}^{2s-1} a_i x^i$ with $a_{2s-1} \not= 0$,
$P_1(x) := x^{\prime 2s}P(\frac{1}{x^\prime})=
\sum_{i=0}^{2s-1} a_i x^{\prime 2s-i}$
is a polynomial with the degree $2s$.
Since $k(x,y,z)=k(x^\prime,y^\prime,z^\prime)$,
the $k$-rationality problem of $k(x,y,z)$ is reduced to
that of $k(x^\prime, y^\prime, z^\prime)$
for the polynomial $P_1(x)$ of even degree.

Since the root of $P_1(x)$ are $0$ and $\{\frac{1}{c_i}\}_{1 \leq i \le r}$,
where $\{c_i\}$ are the roots of $P(x)$,
the conditions $(\mathrm{C1}), \dots, (\mathrm{C5})$ are satisfied for $P_1(x)$ also
(note that we can assume $P(0) \not= 0$
without loss of generality).
When $k$ is a finite field and $|k|$ is small, we may take
a finite extension $k^\prime \supset k$ which satisfies
the conditions $(\mathrm{C1}),\dots, (\mathrm{C5})$,
and continue the argument above.
Note that, if $k(x,y,z)$ is $k$-rational,
then $k^\prime(x,y,z)$ is $k^\prime$-rational.

\subsection{Another biregularization of $T$.} \label{biregT}

In this subsection, after reaching $r$ blowings-up
(this surface is $X_r$ in the subsection \ref{biregularization}),
we shall proceed in another way.
Blow up $X_r$ at the point  $P_{r+1}:x=\infty, u=\infty$ (this surface is $X_{r+1}$,
which is a surface in $X_r \times \mathbb{P}^1$
defined by $\frac{u}{x}=t_{r+1}$),
and then blow-down it by $\widetilde{x=\infty}$.
We denote the obtained surface by $Y_1$.
We lift up and lift down $T: X_r \to X_r$
to get $T:Y_1 \to Y_1$.
$T$ maps $E_{r+1}:x=\infty, u=\infty$ to one point
$P_{r+2} \in E_{r+1}$ defined by $t_{r+1}=\infty$.
So the only fundamental point of $T$ is $P_{r+2}$,
and the only exceptional curve is $E_{r+1}$.
In $\Div(Y_1)$,
$(x=\infty)$ disappears and is replaced by $E_{r+1}$.

Blow up $Y_1$ at $P_{r+2}$
(this surface is in $Y_1 \times \mathbb{P}^1$
defined by $t_{r+2}=\frac{u}{t_{r+1}}=\frac{u}{x^2}$),
and then blow down by $E_{r+1}$.
In the obtained surface $Y_2$,
$E_{r+1}$ disappears and is replaced by $E_{r+2}:x=\infty,t_1=\infty$.
The only fundamental point of $T$ is $P_{r+3} \in E_{r+2}$
defined by $t_{r+2}=\infty$,
and the only exceptional curve is $E_{r+2}$.

For an even $r$, repeat this process $\frac{r}{2}$ times.
On the surface $Y_{\frac{r}{2}}$,
$T$ becomes biregular,
and maps $E_{\frac{3r}{2}}$ biregularly to
$E_{\frac{3r}{2}}$,
as studied in Subsection \ref{biregularization}.
All $E_{r+j}\, (j<\frac{r}{2})$ disappear by the blowings-down.
We shall denote the obtained $Y_{\frac{r}{2}}$ by $Y$.

Thus, for an even $r$, $T$ becomes biregular after $r$ blowings-up
and $\frac{r}{2}$ blowings-up and down.

\subsection{$\Pic(Y)$ as a Galois module.} \label{PicY}

Let $k$ be an algebraically non-closed field,
and $K$ be an algebraic function field with two variables over $k$.
Namely, $K$ is a finite extension of the rational function field with two variables over $k$
such that $k$ is algebraically closed in $K$.

Let $\overline{k}$ be a fixed algebraic closure of $k$.
The $k$-automorphism group of $\overline{k}$ is isomorphic to
$\Gal(k^{\sep}/k)$, where $k^{\sep}$ is the separable closure of $k$,
because every $k$-automorphism of $k^{\sep}$ is extended
uniquely to $\overline{k}$.
$G:=\Gal(k^{\sep}/k)$ acts on
$\overline{k} \otimes_k K$,
assuming that it acts on $K$ trivially,
namely $G \ni \sigma \mapsto \overline{\sigma}=\sigma \otimes \id_K$.

Assume that $\overline{k} \otimes_k K$ is $\overline{k}$-rational,
namely $\overline{k} \otimes_k K=\overline{k}(u,v)$ for some $u,v$.
Let $u^\sigma, v^\sigma$ be the image of $u,v$ by the action of $\overline{\sigma}$,
then we have $\overline{k}(u,v)=\overline{k}(u^\sigma,v^\sigma)$,
so that $u \mapsto u^\sigma, v \mapsto v^\sigma$
induces a $\overline{k}$-automorphism $T_\sigma$
of $\overline{k}(u,v)$.
$T_\sigma$ is different from $\overline{\sigma}$,
because $T_\sigma$ acts trivially on $\overline{k}$.
Let $\widetilde \sigma$ be a $\overline{k}$-automorphism of $\overline{k}(u,v)$
such that $\widetilde \sigma$ acts naturally on $\overline{k}$,
and acts trivially on $u$ and $v$.
Then we have $\overline{\sigma}=T_\sigma \circ \widetilde \sigma$.

$T_\sigma$ induces a birational transformation of $\mathbb{P}^1 \times \mathbb{P}^1$,
while $\widetilde \sigma$ induces a homeomorphic transformation
in Zariski topology of $\mathbb{P}^1 \times \mathbb{P}^1$.
Suppose that after suitable blowings-up or blowings-up and down
of $\mathbb{P}^1 \times \mathbb{P}^1$,
all of $T_\sigma$ become biregular on the obtained surface $Y$.
The lifting of $\widetilde \sigma$ to $Y$ is homeomorphic in Zariski topology.
So the action of $\overline{\sigma}$ induces a permutation of irreducible curves,
and $\Div(Y)$ becomes a permutation $G$-module.

Since the action of $\overline{\sigma}$ keeps the function field
$\overline{k} \otimes_k K=\overline{k}(u,v)$ invariant,
it keeps the principal divisor group invariant,
so taking the factor module,
we see that $\Pic(Y)$ is also a $G$-module.

But since $\Pic(Y)$ is of finite rank as a $\mathbb{Z}$-module,
and since $u,v \in \overline{k}(x,y)$ actually belongs to $l(x,y)$
for some finite extension of $k$,
$\Pic(Y)$ is a $\mathfrak{G}$-module,
where $\mathfrak{G}=\Gal(l/k)$,
$l$ being a sufficiently large finite Galois extension of $k$.

Thus, $\Pic(Y)$ becomes a $\mathfrak{G}$-lattice.
Here a $\mathfrak{G}$-lattice means a free $\mathbb{Z}$-module of finite rank
with the action of $\mathfrak{G}$ as automorphisms.

\subsection{Manin's criterion.} \label{Manin}

Let $K^\prime$ be another algebraic function field with two variables over $k$
such that $\overline{k} \otimes_k K^\prime$ is $\overline{k}$-rational.
Let $\overline{k} \otimes_k K^\prime=\overline{k}(u^\prime,v^\prime)$.
$G=\Gal(k^{\sep}/k)$ acts on $\overline{k} \otimes_k K^\prime$
as $G \ni \sigma \mapsto \overline{\sigma}^\prime=\sigma \otimes \id_{K^\prime}$.

By the discussions in the previous subsection,
$\overline{\sigma}^\prime$ can be written as
$\overline{\sigma}^\prime=T_\sigma^\prime \circ \widetilde{\sigma}^\prime$,
where $T_\sigma^\prime$ is a $\overline{k}$-automorphism of
$\overline{k}(u^\prime,v^\prime)$ and $\widetilde{\sigma}^\prime$
is a $\overline{k}$-automorphism of $\overline{k}(u^\prime,v^\prime)$
which acts on $\overline{k}$ naturally and acts on $u^\prime$ and $v^\prime$ trivially.
$T_\sigma^\prime$ induces a birational transformation of
$\mathbb{P}^1 \times \mathbb{P}^1$.
Suppose that after finite steps of blowings-up (and down),
all $T_\sigma^\prime$ becomes biregular on the obtained surface $Y^\prime$,
so we can regard $\Pic(Y^\prime)$ as a $\mathfrak{G}$-lattice,
where $\mathfrak{G}=\Gal(l/k)$ for sufficiently large finite Galois extension $l$ of $k$.

\begin{prop}
$K$ is $k$-isomorphic to $K^\prime$
if and only if there exists a $\overline{k}$-isomorphism $T$
from $\overline{k} \otimes_k K$ to $\overline{k} \otimes_k K^\prime$
which commutes with the action of $G$,
namely for $\forall \sigma \ni G=\Gal(k^{\sep}/k),
T \circ \overline{\sigma}=\overline{\sigma}^\prime \circ T$.
\end{prop}

\begin{proof}
Suppose that $K$ is $k$-isomorphic to $K^\prime$ and let $T_0$ be the
$k$-isomorphism.
Then $T_0$ is naturally extended to a $\overline{k}$-isomorphism $T$
from $\overline{k} \otimes_k K$ to $\overline{k} \otimes_K K^\prime$,
$T=\id_{\overline{k}} \otimes T_0$.
Evidently $T$ commutes with the action of $G$.

Conversely, suppose that a required $\overline{k}$-isomorphism $T$ exists.
Since $T$ commutes with the action of $G$,
$T$ and $T^{-1}$ map the fixed field of $G$ to each other.
However, the fixed field of $\overline{k} \otimes_k K$
(resp. $\overline{k} \otimes_k K^\prime$) of $G$ is $K$ (resp. $K^\prime$),
and the restriction of $T$ on $K$ becomes a $k$-isomorphism from $K$ to $K^\prime$.
\end{proof}

\begin{defn}
Let $H$ be a finite group and $M$ be an $H$-lattice 
(i.e. a free $\mathbb{Z}$-module of finite rank
with the action of $H$ as automorphisms).
An $H$-lattice $M$ is called permutation if $M$ has a $\mathbb{Z}$-basis
permuted by $H$, i.e. $M \simeq \bigoplus_{1 \leq i \leq m} \mathbb{Z}[H/H_i]$ 
for some subgroups $H_1,\ldots,H_m$ of $H$.
We say that two $H$-lattices $M_1$ and $M_2$ are similar 
if there exist permutation $H$-lattices $P_1$ and $P_2$ such that 
$M_1\oplus P_1\simeq M_2\oplus P_2$. 
\end{defn}

\begin{prop}[Manin {\cite{Man67}}]
Let $\Pic(Y)$ (resp. $\Pic(Y^\prime)$) be the $\mathfrak{G}$-lattice
corresponding to $K$ (resp. $K^\prime$) as in Subsection \ref{PicY}.
If $K$ is $k$-isomorphic to $K^\prime$, then
$\Pic(Y)$ and $\Pic(Y^\prime)$ are similar, i.e. there exist
permutation $\mathfrak{G}$-lattices $P_1$ and $P_2$ such that
$\Pic(Y) \oplus P_1 \simeq \Pic(Y^\prime) \oplus P_2$.
\end{prop}

\begin{proof}
Assume the existence of a required $\overline{k}$-isomorphism $T$
from $\overline{k}(u,v)$ to $\overline{k}(u^\prime,v^\prime)$.
Then $T$ induces a birational transformation of $\mathbb{P}^1 \times \mathbb{P}^1$.
After suitable blowings-up or blowings-up and down,
$T$ is lifted to a birational map from $Y$ to $Y^\prime$.
Though it may not be biregular on $Y$,
after further suitable blowings-up,
we can reach the surfaces $Z$ and $Z^\prime$,
on which $T$ (and $T^{-1}$) becomes biregular.

Since $T$ is biregular, we have $\Pic(Z) \simeq \Pic(Z^\prime)$
as $\mathbb{Z}$-modules.
Since $T$ commutes with the action of $G$,
(then their liftings commutes also),
$\Pic(Z) \simeq \Pic(Z^\prime)$ as $\mathfrak{G}$-lattices also.

Only remained to prove is that $\Pic(Z) \simeq \Pic(Y) \oplus P$
for some permutation $\mathfrak{G}$-lattice $P$.

Let $\{E_j\}$ be the successive blowings-up to reach $Z$ from $Y$.
Since $T$ commutes with the action of $G$,
the set of fundamental points of $T$ is $G$-invariant,
and the action of $G$ induces permutations of $\{E_j\}$.

Let $\{e_i\}$ be the basis of $\Pic(Y)$ as a free $\mathbb{Z}$-module.
Then $\Pic(Z)$ is a free $\mathbb{Z}$-module
with the basis $\{\pi^\ast e_i\} \cup \{ E_j \}$,
where $\pi^\ast$ is a $\mathbb{Z}$-linear map
from $\Pic(Y)$ to $\Pic(Z)$,
obtained by the iteration of $\pi^\ast$ mentioned at the end of Subsection \ref{intersection}.
Let $M_1$ (resp. $M_2$) be a free $\mathbb{Z}$-module with the basis
$\{\pi^\ast e_i\}$ (resp. $\{E_j\}$).
Then $\Pic(Z) \simeq M_1 \oplus M_2$ as $\mathbb{Z}$-modules.
However, $M_2$ is a permutation $\mathfrak{G}$-lattice
as mentioned above.
We can show that $M_1$ is also a $\mathfrak{G}$-lattice
which is isomorphic to $\Pic(Y)$.
\end{proof}

\begin{cor}
If $K$ is $k$-isomorphic to $K^\prime$, then
$H^1(\mathfrak{G},\Pic(Y)) \simeq H^1(\mathfrak{G},\Pic(Y^\prime))$
and $\widehat H^{-1}(\mathfrak{G},\Pic(Y)) \simeq
\widehat H^{-1}(\mathfrak{G},\Pic(Y^\prime))$,
where $H^1$ is Galois cohomology
and $\widehat H^{-1}$ is Tate cohomology.
\end{cor}

This comes from $H^1(\mathfrak{G},P)=\widehat H^{-1}(\mathfrak{G},P)=0$
for a permutation $\mathfrak{G}$-lattice $P$.

Especially, if $K$ is $k$-rational,
then $H^1(\mathfrak{G},\Pic(Y))=\widehat H^{-1}(\mathfrak{G},\Pic(Y))=0$
by the following reason.

The $k$-rationality of $K$ means that $K$ is $k$-isomorphic to
the two dimensional rational function field $K^\prime=k(x,y)$.
In this case,
$\overline{k} \otimes_k K^\prime=\overline{k}(x,y)$
and $\overline{\sigma}^\prime$ acts trivially on $x$ and $y$.
So, $Y^\prime=\mathbb{P}^1 \times \mathbb{P}^1$
and $\Pic(Y^\prime)$ is a trivial $\mathfrak{G}$-lattice,
so that $H^1(\mathfrak{G},\Pic(Y^\prime))=\widehat H^{-1}(\mathfrak{G},\Pic(Y^\prime))=0$.
In other words, $H^1(\mathfrak{G},\Pic(Y)) \not=0$
or $\widehat H^{-1}(\mathfrak{G},\Pic(Y)) \not= 0$
is a criterion for the $k$-irrationality of $K$.

\subsection{\bf Proof of Theorem \ref{A}} \label{pfA}
As mentioned in Subsection \ref{idea}, Theorem \ref{A} was proved already by
Colliot-Th\'el\`ene and Sansuc \cite[Proposition 1 (v)]{San81}
(see also \cite{CTS94} ). The following proof is included for the convenience
of the reader.

Let $K$ be the quadratic extension of $k(x,y)$ defined by
$(\ref{cha})$: $z^2=ay^2+P(x)$ with conditions $(\mathrm{C1}),\dots, (\mathrm{C5})$
in Subsection \ref{main}. Then $\overline{k} \otimes_k
K=\overline{k}(x,u)$ where $u=z+\sqrt{a}y$. For $\sigma \in
G=\Gal(k^{\sep}/k)$, $T_\sigma$ is either the identity or equal to
$T: x \mapsto x, u \mapsto \frac{P(x)}{u}$ according to whether
$\sqrt{a}$ is invariant by $\sigma$ or not. $T$ induces a
birational transformation of $\mathbb{P}^1 \times \mathbb{P}^1$,
and it becomes biregular on the obtained
surface $Y$, mentioned in Subsection \ref{biregT}.
Then $\Pic(Y)$ is a free $\mathbb{Z}$-module of rank $r+2$ with
the basis $E_i (1 \leq i \leq r)$, $F$ and $F^\prime$.
(We assume that $r$ is even.)

We shall determine the structure of $\Pic(Y)$ as a $\mathfrak{G}$-lattice,
where $\mathfrak{G}=\Gal(l/k)$,
$l$ being the splitting field of $P(x)$ over $k$.

As studied in Subsection \ref{biregT},
$T$ maps $E_i$ to $(\widetilde{x=c_i})$
for $1 \leq i \leq r$,
and $(\widetilde{u=c})$ to $(\widetilde{u=\frac{P(x)}{c}})$.
The next question is what divisor classes
they belong to.

Let $\pi^\ast$ and $\rho^\ast$ be a $\mathbb{Z}$-linear map from $\Div(\mathbb{P}^1 \times \mathbb{P}^1)$
to $\Div(Y)$,
obtained by the iteration of $\pi^\ast$ mentioned at the end of Subsection \ref{intersection} and Subsection \ref{updown}.
Then we have
\begin{eqnarray}
\pi^\ast(x=c_i) &=& (\widetilde{x=c_i})+E_i \text{ for } 1 \leq i \leq r, \nonumber \\
\rho^\ast(u=\frac{P(c)}{c}) &=& (\widetilde{u=\frac{P(c)}{c}})+\sum_{j=1}^{r}E_j+\frac{r}{2}F.
\end{eqnarray}
So that in $\Pic(Y)$, we have
\begin{eqnarray}
(\widetilde{x=c_i}) &\equiv& F-E_i, \\
(\widetilde{u=\frac{P(x)}{c}}) &\equiv& F^\prime+\frac{r}{2}F-\sum_{j=1}^r E_j. \nonumber
\end{eqnarray}
Therefore, the action of $\overline{\sigma}=T_\sigma \circ \widetilde{\sigma}$
on $\Pic(Y)$ is represented by the following matrix $g_\sigma$,
with $E_1, E_2, \dots, E_r, F, F^\prime$
as the basis in this order.
\begin{eqnarray} \label{even1}
\text{For } \sigma \in N=\Gal(l/k(\sqrt{a})), &
g_\sigma=
\begin{pmatrix}
A_\sigma & 0 \\
0 & I_2 \\
\end{pmatrix}, \\
\text{for } \sigma \in \mathfrak{G} \setminus N, &
g_\sigma=
\begin{pmatrix}
-A_\sigma & 1 & 0 \\
0 & 1 & 0 \\
-1 & \frac{r}{2} & 1
\end{pmatrix}, \nonumber
\end{eqnarray}
where  $1$ (resp. $0$, $-1$) stands for the matrix whose entries are all $1$
(resp. $0$,$-1$).

$A_\sigma$ is the permutation matrix of the permutation of $\{c_i\}$
induced by $\sigma$.
Suppose that $P(x)$ is a product of $r^\prime$ irreducible polynomials.
Then, the set of roots $\{c_i\}$ of $P(x)$ is divided into $r^\prime$ blocks,
each of which consists of the roots of the same irreducible component.
Each block is a transitive part by the action of $\mathfrak{G}$.
Since each irreducible component is assumed to be irreducible
also over $k(\sqrt{a})$,
the action of $N$ is also transitive on each block.
The block is called even (resp. odd),
when the degree of the corresponding irreducible polynomial is even (resp. odd).

Let $M_0$ be the submodule spanned by
$\{E_i | 1 \leq i \leq r\}$.
An element of $M_0$ is written as $\sum_{i=1}^r a_iE_i$, $a_i \in \mathbb{Z}$.
Let $s_j$ be the sum of $a_i$
when $i$ runs over the $j$-th block.
Let $M_e$ be the submodule of $M_0$,
consisting of elements such that
$\sum_{j=1}^{r^\prime} s_j$ is even.
Let $M_b$ be the submodule of $M_0$,
consisting of elements such that $s_j$ is even for every $j$.
We have $M_0/M_e \simeq \mathbb{Z}/2\mathbb{Z},
M_0/M_b \simeq (\mathbb{Z}/2\mathbb{Z})^{r^\prime}$ and
$M_e/M_b \simeq (\mathbb{Z}/2\mathbb{Z})^{r^\prime-1}$,
where $r^\prime$ is the number of the blocks.

By definition, we have
$\widehat H^{-1}(\mathfrak{G},\Pic(Y))=Z/B$,
where $Z$ and $B$ are submodules of $\Pic(Y)$ defined by
\begin{equation}
Z=\ker(\sum_{\sigma} g_\sigma), \quad
B \text{ is the module spanned by }
\bigcup_{\sigma} \im(\sigma-\id),
\end{equation}
where the summation and the union are taken over $\sigma \in \mathfrak{G}$.

The sum $\sum g_\sigma$ is zero
except the last row and the $(r+1)$-th column,
so that the rank of $Z$ is $r$ and the projection to $M_0$
(projection as $\mathbb{Z}$-modules) is injective.
Let $Z^\prime$ and $B^\prime$ be the images of the projection of $Z$ and $B$
respectively,
then $Z/B \simeq Z^\prime/B^\prime$.

We see that $Z^\prime=M_e$ and $B^\prime=M_b+(\sum_{i=1}^r E_i)\mathbb{Z}$.
Since $M_e/M_b \simeq (\mathbb{Z}/2\mathbb{Z})^{r^\prime-1}$,
and since $\sum_{i=1}^r E_i \in M_b$
if and only if odd block does not exist,
we have
\begin{eqnarray}
\widehat H^{-1}(\mathfrak{G},\Pic(Y)) \simeq \begin{cases}
(\mathbb{Z}/2\mathbb{Z})^{r^\prime-1}
& {\text{ if odd block does not exist,}} \\
(\mathbb{Z}/2\mathbb{Z})^{r^\prime-2}
& {\text{ if odd blocks exist.}}
\end{cases}
\end{eqnarray}
As for $H^1(\mathfrak{G},\Pic(Y))$,
we proceed as follows.
In general for a $\mathfrak{G}$-lattice $M$,
$H^1(\mathfrak{G},M)$ is isomorphic to $\widehat H^{-1}$
of the dual lattice $M^\prime$.
So that as for $H^1(\mathfrak{G},\Pic(Y))$,
it suffices to calculate $\widehat H^{-1}$
for the transposed matrix of $(\ref{even1})$.
The calculation shows that
$H^1(\mathfrak{G},\Pic(Y)) \simeq \widehat H^{-1}(\mathfrak{G},\Pic(Y))$,
though the matrix $(\ref{even1})$ is not symmetric.

Thus, the proof of Theorem \ref{A} has been completed.
Note that when $\deg P$ is odd,
$\widehat H^{-1}(\mathfrak{G},\Pic(Y)) \simeq (\mathbb{Z}/2\mathbb{Z})^{r^\prime-1}$,
because the reduction in Subsection \ref{degeven} implies that $r^\prime$ increases by $1$
and odd block obviously exist,
so we have $(r^\prime+1)-2=r^\prime-1$.


\subsection{$\mathfrak{G}$-invariant classes} \label{Ginv}

From the matrices $(\ref{even1})$,
we have the followings.

\begin{prop} \label{prop10}
The submodule $\Pic(Y)^\mathfrak{G}$ of $\mathfrak{G}$-invariant classes
has rank 2 with the basis $F$ and $\Omega$.
We have $F \cdot F=0, F \cdot \Omega=-2$ and $\Omega \cdot \Omega=8-r$.
\end{prop}

\begin{proof}
From $(\ref{even1})$, $\sum_\sigma g_\sigma$ is represented by
\begin{equation} \label{matsg}
\sum_\sigma g_\sigma=\frac{|\mathfrak{G}|}{2}
\begin{pmatrix}
O_r & 1 & 0 \\
0 & 2 & 0 \\
-1 & \frac{r}{2} & 2
\end{pmatrix}
\end{equation}
with the basis $E_i$ and $F, F^\prime$,
where $O_r$ is the $r \times r$ matrix
whose entries are all zero.

Since the image of $\sum_\sigma g_\sigma$ is contained in $\Pic(Y)^\mathfrak{G}$
with finite index, we see that $\Pic(Y)^\mathfrak{G}$ is generated by $F$
and $2F^\prime+\frac{r}{2}F-\sum_{i=1}^rE_i$.
But since $\Omega_Y=-2F^\prime+(\frac{r}{2}-2)F+\sum_{i=1}^rE_i$
as stated in Subsection \ref{pre2},
$\Pic(Y)^\mathfrak{G}$ is generated by $F$ and $\Omega$.

$F \cdot F=0$ etc. is easily obtained by the intersection form of $Y_{rs}$ stated before.
\end{proof}

Note that if a curve $C$ is $\mathfrak{G}$-invariant,
then its class is also $\mathfrak{G}$-invariant.
The converse is not true,
because $C$ can be moved in the same class.


\subsection{\bf Proof of Theorem \ref{B}} \label{pfB}

If $k(x,y,z)$ is $k$-rational, write $k(x,y,z)=k(t,s)$ for some
$t$ and $s$. Thus $\overline{l}(x,u)=\overline{l}(t,s)$
with $\mathfrak{G}$-invariant $t,s$.

Let $Y$ be the algebraic surface obtained in
Subsection \ref{biregT}. It follows that $Y$ is
$k$-birational with $\mathbb{P}^1 \times \mathbb{P}^1$. Here
``$k$-birational'' means that there exists a birational mapping
$\mathbb{P}^1 \times \mathbb{P}^1 \rightarrow Y$ which commutes
with the action of $\mathfrak{G}$, where $\mathfrak{G}$ acts trivially on
$t$ and $s$.

Let $\Phi$ be a $k$-birational mapping $\mathbb{P}^1 \times \mathbb{P}^1
\rightarrow Y$.
After finite steps of blowings-up of $\mathbb{P}^1 \times \mathbb{P}^1$ and $Y$
respectively, $\Phi$ is lifted to a biregular mapping
$Z \rightarrow Z^\prime$.

For $a,b \in k$,
the lines $t=a$ and $s=b$ are $\mathfrak{G}$-invariant in
$\mathbb{P}^1 \times \mathbb{P}^1$,
so that their images are also $\mathfrak{G}$-invariant in $Y$ or in $Z^\prime$.
Suppose that $t=a$ is not an exceptional curve of $\Phi$
and does not pass through a fundamental point of $\Phi$,
then the values of intersection form
$(t=a) \cdot (t=a)=0, (t=a) \cdot \Omega=-2$
are kept invariant under the blowings-up,
so the image $C$ in $Z^\prime$ also satisfies
$C \cdot C=0$ and $C \cdot \Omega=-2$ in $Z^\prime$.
\hfill\break
$Z^\prime$ is obtained from $Y$ by successive blowings-up
$\{E_j^\prime\}$.
By each blowing-up, $C \cdot C$ is decreased by $(C \cdot E_j^\prime)^2$
and $\Omega \cdot C$ is increased by $C \cdot E_j^\prime$,
by $(\ref{2_1})$ and $(\ref{widetildeC})$.

Thus we have on the surface $Y$
\begin{eqnarray} \label{CC}
C \cdot C=\sum_j m_j^2, \\
\Omega \cdot C=-2-\sum_j m_j, \nonumber
\end{eqnarray}
where $m_j=C \cdot E_j^\prime$.
We represent the class of $C$ with $\nu F-m\Omega$, then we have
\begin{eqnarray} \label{CC2}
C \cdot C &=& 4m\nu+\omega m^2, \\
\Omega \cdot C &=& -2\nu-m\omega
\text{ where } \omega=\Omega \cdot \Omega. \nonumber
\end{eqnarray}
Combining $(\ref{CC})$ with $(\ref{CC2})$, we get
\begin{eqnarray} \label{meq}
\sum_j m_j^2=4m\nu+\omega m^2, \\
\sum_j m_j=2\nu+m\omega-2. \nonumber
\end{eqnarray}
From $C \cdot F=2m$, we have $m \geq 0$,
and $m=0$ means that $C$ is $(x=c)$ for some $c \in k$.
For $m>0$, we have $0 \leq m_j \leq 2m$
by the following reason.

Let $(x=c_j)$ be the line passing through the base point of $E_j^\prime$.
$\widetilde C \cdot (\widetilde{x=c_j}) \geq 0$
on the blowing-up surface implies $C \cdot E_j^\prime \leq C \cdot (x=c_j)
=C \cdot F=2m$.
Note that for successive blowings-up $\{E_j^\prime\}$
(namely $E_1^\prime$ is the blowing-up at $P_1 \in Y$,
$E_2^\prime$ is the blowing-up at $P_2 \in E_1^\prime$,
$E_3^\prime$ is the blowing-up at $P_3 \in E_2^\prime$,
and so on),
$C \cdot E_j^\prime$ is monotone decreasing.
Any successive blowing-up also satisfies
$C \cdot E_j^\prime \leq 2m$.

Thus we have $0 \leq \sum_j m_j^2 \leq 2m\sum_j m_j$,
so that
\begin{equation}
4m\nu+\omega m^2 \leq 2m(2\nu+m\omega-2)
\end{equation}
which yields $\omega m^2 \geq 4m$.
This is impossible if $\omega \leq 0$ and $m > 0$.
Since $\omega=8-r$,
if $r \geq 8$,
then any $\mathfrak{G}$-invariant curve other than $x=$const.
cannot become the image of $t=a$.

The same holds for $s=b$. Since $t$ and $s$ are algebraically
independent, at least one of $t$ and $s$ depends on $u$ so that $m
\geq 1$. Thus we reach a contradiction.
\hfill $\square$


\subsection{Reduction to a del Pezzo surface} \label{deg46}

Since we are assuming that $r$ is even,
the remained cases are $r=6$ and $r=4$.
We shall continue the discussion for these cases.

Let $P_1^\prime$ be a point on $Y$ such that $m_1 > m$.
The point $P_1^\prime$ lies on $(x=c)$ for some $c$.
Note that $c \not= c_i$,
since $C \cdot E_i=C \cdot \Omega=m$,
so that for any point $P$ on $(x=c_i)$,
$C \cdot E_P >m$ can not occur.

Let $Y_1$ be the blowing-up of $Y$ at $P_1^\prime$,
and $Y_1^\prime$ be the blowing-down of $Y_1$
by $x=c$.
$Y_1$ is biregular with the blowings-up of $Y_1^\prime$
at some $Q_1$, where
\begin{equation}
C \cdot E_{Q_1} =C \cdot F-C \cdot E_1^\prime=2m-m_1<m.
\end{equation}

Let $P_2^\prime$ be a point on $Y_1$ such that $m_2 >m$.
The point $P_2^\prime$ does not lie on $(x=c)$,
because of $C \cdot E_1^\prime >m$.

Since $Y_1 \setminus (x=c)$ is biregular
with $Y_1^\prime \setminus \{Q_1\}$,
the blowing-up at $P_2^\prime$ induces the blowing-up of $Y_1^\prime$
at the correspoinding point $P_2^{\prime\prime}$.
After the blowing-up at $P_2^{\prime\prime}$,
blow-down by $(x=c^\prime)$
passing through $P_2^{\prime\prime}$.
Let $Y_2^\prime$ be the obtained surface.

Repeat this procedure until all $E_j^\prime$ with $m_j>m$ are eliminated.
We shall denote the obtained surface
by $Y_C$
($Y_C$ depends on $C$).
Then the successive blowings-up of $Y$ at $\{P_j^\prime\}$
is obtained by the successive blowings-up of $Y_C$ as explained in Subsection \ref{updown},
but this time every blowing-up
satisfies $\mu_j := C \cdot E_{Q_j} \leq m$.

On $Y_C$, we have $C \equiv \nu^\prime F-m\Omega$
where $\nu^\prime=\nu-\sum^\prime (m_j-m)$.
Here the summation $\sum^\prime$ is taken over such $j$ that $m_j>m$.

The self intersection number $C \, \cdot \, C$ on $Y_C$ decreases from that on $Y$ by
$4m \sum^\prime (m_j-m)$, and we have
\begin{eqnarray}
\sum_j \mu_j^2=4\nu^\prime m+m^2\omega, \\
\sum_j \mu_j=2\nu^\prime+m\omega-2, \nonumber
\end{eqnarray}
where $\mu_j=\min(m_j,2m-m_j) \leq m$.

This time $4\nu^\prime m+m^2\omega \leq
m (2\nu^\prime+m\omega-2)$ yields $\nu^\prime \leq -1$.
Since $C \cdot C \geq 0$ implies $\nu^\prime \geq -\frac{m\omega}{4}$,
we must have $\frac{m\omega}{4} \geq 1$,
namely $m \geq \frac{4}{\omega}$.

\begin{prop}
The surface $Y_C$ is a del Pezzo surface.
\end{prop}

\begin{proof}
We must check only $\Omega \cdot \Gamma<0$.

Since $\Omega \cdot F=-2$ and $\Omega \cdot C \leq -2$,
we can suppose that $\Gamma \not\in F$ and $\Gamma \not= C$.
Then $\Gamma \cdot C \geq 0$ implies
$\nu^\prime F \cdot \Gamma-m \Omega \cdot \Gamma \geq 0$,
namely $\nu^\prime F \cdot \Gamma \geq m\Omega \cdot \Gamma$.
Since $\nu^\prime \leq -1$, we have $\Omega \cdot \Gamma \leq 0$,
and $\Omega \cdot \Gamma=0$ is possible only when $F \cdot \Gamma=0$.

Let $\Gamma_\sigma$ be the image of $\Gamma$ by the action of $\sigma \in \mathfrak{G}$.
Then $\sum_\sigma \Gamma_\sigma$ is $\mathfrak{G}$-invariant.
Since the intersection form is kept by the action of $\mathfrak{G}$,
we have $F \cdot \sum_\sigma \Gamma_\sigma=$
\break
$\Omega \cdot \sum_\sigma \Gamma_\sigma=0$,
so $\sum_\sigma \Gamma_\sigma \equiv 0$ in $\Pic(Y_C)$.
But any principal divisor can not be an integral divisor
(= positive linear combination of irreducible curves),
so this is impossible, and the proposition has been proved.
\end{proof}

\subsection{Further blowings-up and down to reach a contradiction.} \label{updown2}

For $r=6$ or $r=4$ (namely for $\omega=2$ or $\omega=4$),
let $F_1$ be $-\frac{4}{\omega}\Omega-F$ mentioned in Theorem \ref{lem2} (2).

We can choose $F_1$ and $\Omega$ as the basis of $\Pic(Y_C)^\mathfrak{G}$,
and we have
$$C \equiv \nu^\prime F-m \Omega=
-\nu^\prime F_1- \big(m+\frac{4\nu^\prime}{\omega}\big)\Omega.$$
Put $m_1=m+\frac{4\nu^\prime}{\omega}$,
then $-1 \geq \nu^\prime \geq -\frac{m\omega}{4}$
yields $0 \leq m_1 < m$.

The case $m_1=0$ is discarded by the following reason.
$m_1=0$ implies $C \cdot C=0$,
so $\mu_j=0$ for any $j$.
Let $C^\prime$ be the image of $s=b$,
then we have $C \cdot C^\prime=1$ on $Z^\prime$,
but since $\mu_j=0$,
we have $C \cdot C^\prime=1$ also on $Y_C$.
On the other hand, $F_1 \cdot F_1=0$ and
$F_1 \cdot \Omega=-2$ imply that $F_1 \cdot D$ is even
for any $D \in \Pic(Y_C)^\mathfrak{G}$.
This is a contradiction.

For $E_j^{\prime\prime}$ such that $\mu_j>m_1$,
after the blowing-up $E_j^{\prime\prime}$,
blow down by the curve $\Gamma_j$
which belongs to $F_1$ and passes through $P_j$
(=base point of $E_j^{\prime\prime}$).
(Such a curve $\Gamma_j$ exists by Theorem \ref{lem2} (1).)
On the obtained surface $X_1$, we have
$$C \equiv \nu_1 F_1-m_1 \Omega \, (\nu_1 \leq -1).$$
Repeat this procedure.
Since $m>m_1>m_2>\cdots$ is monotone decreasing,
after finite steps, we reach $m_l=0$ or $1 \leq m_l < \frac{4}{\omega}$.
In the former case,
$C \equiv F_l$ on $X_l$,
so $C^\prime \cdot C=1$ is impossible for any other $C^\prime \in \Pic(X_l)^\mathfrak{G}$.
In the latter case, we can not reach $C \cdot C=0, C \cdot \Omega=-2$
by further blowings-up.

Anyway, $k(x,y,z)$ can not be $k$-rational.


\section{Conic bundles.} \label{chCo}

\subsection{Preliminaries.} \label{pre1}

First, recall that the function field of a conic bundle over $\mathbb{P}_k^1$
as in (\ref{cha2})
may be written as $K=k(x,y,z)$ with a relation
$z^2=Py^2+Q$ where $P, Q$ are some non-zero separable polynomials
in $k[x]$.
As before, we assume that $\ch k \not=2$.

Note that the rationality problem for the pair $(P,Q)$ is equivalent with
that for $(Q,P)$,
because by putting $z=yz^\prime$ and $y=\frac{1}{y^\prime}$,
$z^2=Py^2+Q$ is rewritten as $z^{\prime2}=P+Qy^{\prime2}$.
It is equivalent also with that of $(P,Q^\prime)$
where $Q^\prime=Q(F^2-PG^2), F,G \in k[x]$.
By putting $z=Fz^\prime+PGy^\prime,
y=Gz^\prime+Fy^\prime,
z^2-Py^2=Q$ is rewritten as
$(F^2-PG^2)(z^{\prime2}-Py^{\prime2})=Q$.
When $F=0, G=1$, it is equivalent with that for $(P,-PQ)$.

When $\deg P=0$ or $\deg Q=0$ or $P/Q=$const.,
the equation is reduced to the previous section.
Thus we will assume $\deg P \geq 1$, $\deg Q \geq 1$,
and $P/Q \not= c \, (c \in k)$.

\begin{prop} \label{prop41}
If there exist $A(x), B(x), C(x) \in k[x]$
such that
\begin{equation} \label{ABC}
A^2P+B^2Q=C^2,
\end{equation}
then $k(x,y,z)$ is $k(x)$-rational
(note that none of $A, B, C$ is zero under the assumption above).
\end{prop}

\begin{proof}
From $(\ref{ABC})$,
$z^2=Py^2+Q$ is rewritten as
$B^2z^2=B^2Py^2+C^2-A^2P$,
namely as $B^2z^2-C^2=P(B^2y^2-A^2)$.
Put $Bz+C=z^\prime$ and $By+A=y^\prime$,
then we have $z^\prime(z^\prime-2C)=Py^\prime(y^\prime-2A)$.
Put $z^\prime=uy^\prime$,
then $u(u-\frac{2}{y^\prime}C)=P(1-\frac{2}{y^\prime}A)$,
so $y^\prime \in k(x,u)$.
This implies that $k(x,y,z)=k(x,y^\prime,z^\prime)
=k(x,y^\prime,u)=k(x,u)$, so Proposition \ref{prop41} holds.
Explicitly writing, we have
\begin{equation} \label{ABC2}
y=\frac{-Au^2+2Cu-AP}{B(u^2-P)}, \,
z=\frac{Cu^2-2APu+CP}{B(u^2-P)}.
\end{equation}
\end{proof}

\begin{prop} \label{prop42}
For a sufficiently large Galois extension $l$ of $k$,
$l(x,y,z)$ is $l(x)$-rational.
\end{prop}

\begin{proof}
It suffices to show the existence of $A,B$ and $C$ in $k^{\sep}[x]$
satisfying $(\ref{ABC})$ where $k^{\sep}$ is the separable closure of $k$;
that is, we will show that the Hilbert norm-residue symbol $(P,Q)_2$ over the field
$k^{\sep}(x)$ is trivial.

Let $\overline{k}$ be a fixed algebraic closure of $k$.
Thus $\overline{k}$ is a purely inseparable extension of $K^{\sep}$.

If $\ch k=0$, then $\overline{k}=k^{\sep}$.
By Tsen's theorem, the field $\overline{k}(x)$ is a $C_1$-field
\cite[page 22]{Gre69}.
Hence there are polynomials $A, B, C \in \overline{k}[x]$
such that $C \not= 0$ and
$(A/C)^2P+(B/C)^2Q=1$.

Now suppose that $\ch k=p>0$.
Remember that $p \not= 2$.
Let $\mathfrak{A}$ be the quaternion algebra over $k^{\sep}(x)$
corresponding to the Hilbert norm-residue symbol $(P,Q)_2$
over the field $k^{\sep}(x)$.
Since the Brauer group $\Br\big(\overline{k}(x)\big)=0$ by another theorem
of Tsen \cite[page 4]{Gre69},
$\mathfrak{A}$ is split by some finite purely inseparable extension of $k^{\sep}(x)$.
Thus $p^n[\mathfrak{A}]=0$ for some non-negative integer $n$
where $[\mathfrak{A}]$ denotes the similarity class of $[\mathfrak{A}]$
in the Brauer group.
Because $\mathfrak{A}$ is a quaternion algebra,
it is necessary that $2[\mathfrak{A}]=0$.
Thus $[\mathfrak{A}]=0$ and the Hilbert norm-residue symbol
$(P,Q)_2$ is trivial.
\end{proof}

\begin{prop} \label{prop43}
If $A, B, C \in l[x]$ satisfy $(\ref{ABC})$,
then for any $f \in l[x]$,
the following $A_1, B_1, C_1$ also satisfy $(\ref{ABC})$:
\begin{eqnarray} \label{ABC3}
A_1 &=& A(f^2-Q), \nonumber \\
B_1 &=& Bf^2+2Cf+BQ, \\
C_1 &=& Cf^2+2BQf+CQ. \nonumber
\end{eqnarray}
\end{prop}

\begin{proof}
Regarding $A_1, B_1, C_1$ as quadratic polynomials of $f$,
the comparison of the coefficients of $A_1^2P+B_1^2Q$ and $C_1^2$
leads to the proof.
\end{proof}

\begin{prop} \label{prop44}
$A, B, C \in l[x]$ in {\rm Proposition \ref{prop43}}
can be chosen as the following conditions hold:
\begin{sub}
\item[$(1)$]
Any two of $A, B, C$ are mutually disjoint;
\item[$(2)$]
$\deg B$ is sufficiently large;
\item[$(3)$]
$B$ is disjoint with $Q$;
\item[$(4)$]
all zeros of $B$ are simple;
\item[$(5)$]
$b_0=1$ where $b_0$ is the coefficient of the highest degree term of $B$.
\end{sub}
\end{prop}

\begin{proof}
\hfill\break
$(1)$
Dividing by $\gcd(A, B, C)$,
we can assume that $A, B, C$ are mutually disjoint.
Then any two of them are already mutually disjoint,
because a common zero of two of them is necessarily a zero
of the third one.
Note that all zeros of $P, Q$ are simple.
\hfill\break
$(2)$
Suppose that $A, B, C$ satisfy $(1)$.
Let $f=\alpha x^n$.
Then as shown below,
$B_1$ and $C_1$ in $(\ref{ABC3})$ are mutually disjoint except finite number of $\alpha$.
Taking $n$ so large,
we get $A_1, B_1, C_1$ which satisfy (1) and (2).
$B_1$ and $C_1$ are mutually disjoint if their resultant is not zero.
The resultant is a polynomial of $\alpha$
(with fixed $B$, $C$, $Q$),
so the number of zeros is finite
unless the resultant is identically zero.
\hfill\break
$(3)$
Suppose that $A, B, C$ satisfy $(1)$ and $(2)$.
Let $f=\alpha \in k^\times$.
Then $B_1$ is disjoint with $Q$ except finite number of $\alpha$.
If $B(c)=Q(c)=0$,
then $B_1(c)=2\alpha C(c) \not= 0$ for $\alpha \not= 0$.
If $B(c) \not= 0, Q(c)=0$,
then $B_1(C) \not= 0$ for
$\alpha \not=0, \alpha \not= -\frac{2C(c)}{B(c)}$.
This proves (3).
\hfill\break
$(4)$
The above $B_1$ has no multiple zero except finite number of $\alpha$.
The proof is similar with that of (2).

We can check that the resultant of $B_1$ and $B_1^\prime$
($=$ the derivative of $B_1$) is not identically zero
as a polynomial of $\alpha$.

Finally dividing by a constant,
we can set $b_0=1$.
This is the claim of (5)
\end{proof}

Hereafter we shall always assume the conditions $(1), \dots, (5)$
of Proposition \ref{prop44}
for $A, B, C$.

\subsection{Biregularization of $T$.}

Let $A, B, C$ and $A_1, B_1, C_1$ be as in Proposition \ref{prop44}.
(Note that $A_1, B_1, C_1$ are arbitrary,
and may not be in the form of $(\ref{ABC3})$.)

\begin{prop} \label{prop3}
Let $u=\frac{Bz+C}{By+A}$ and $u_1=\frac{B_1z+C_1}{B_1y+A_1}$.
Then we have
\begin{equation} \label{ABC4}
u_1=\frac{Du+EP}{Eu+D}
\end{equation}
where $D,E \in l[x]$ are mutually disjoint and
$\frac{D}{E}=\frac{B_1C+BC_1}{A_1B-AB_1}=
\frac{(A_1B+AB_1)P}{BC_1-B_1C}.$
\end{prop}

\begin{proof}
Since $l(x,y,z)=l(x,u)=l(x,u_1)$,
$u_1$ should be a linear fraction of $u$
with $l(x)$-coefficients.

When $u=\infty$, from $(\ref{ABC2})$ we have $y=-\frac{A}{B}, z=\frac{C}{B}$
so that $u_1=\frac{B_1C+BC_1}{A_1B-AB_1}$.
When $u=0$, we have $y=\frac{A}{B}, z=-\frac{C}{B}$
so that $u_1=\frac{-B_1C+BC_1}{A_1B+AB_1}$.
Since $(B^2A_1^2-B_1^2A^2)P=B ^2C_1^2-B_1^2C^2$,
we have $\frac{-B_1C+BC_1}{A_1B+AB_1}=
\frac{(A_1B-AB_1)P}{BC_1+B_1C}$.
From these facts, we obtain
$u_1=\frac{Du+EP}{Eu+D}$
where
\begin{equation}
E \text{ and } D \text{ are mutually disjoint and }
\frac{D}{E}=\frac{B_1C+BC_1}{A_1B-AB_1}=
\frac{(A_1B+AB_1)P}{BC_1-B_1C}.
\end{equation}
\end{proof}

Note that at least one of $\frac{B_1C+BC_1}{A_1B-AB_1}$
and $\frac{(A_1B+AB_1)P}{BC_1-B_1C}$ is not $\frac{0}{0}$
(i.e. at least one of
$B_1C+BC_1$, $A_1B-AB_1$,
$(A_1B+AB_1)P$ and $BC_1-B_1C$
is not zero).

Let $T$ be the birational mapping
\begin{equation} \label{T}
T: \mathbb{P}^1 \times \mathbb{P}^1
\rightarrow \mathbb{P}^1 \times \mathbb{P}^1,
T: x \mapsto x, u \mapsto \frac{Du+EP}{Eu+D}.
\end{equation}
Then $T=\id$ if and only if $E=0$,
so if and only if $A=A_1, B=B_1, C=C_1$
under the assumptions (1) and (5) of Proposition \ref{prop44}.
In the following discussions,
we shall investigate how to blow-up
$\mathbb{P}^1 \times \mathbb{P}^1$ to make $T$ biregular.

Since $T$ keeps $x$ unchanged,
an exceptional curve of $T$ is in the form of $x=c$.

\begin{prop} \label{prop4}
Let $T$ be the map as in $(\ref{T})$.
Then $x=c \, (c \not= \infty)$ is an exceptional curve of $T$
only if $c$ is a zero of $BB_1PQ$.
\end{prop}
\begin{proof}
$x=c$ is an exceptional curve of $T$
if and only if $c$ is a zero of $D^2-E^2P$.
However, $D^2-E^2P$ divides both of
$(BC_1+B_1C)^2-(A_1B-AB_1)^2P$ and
$(A_1B+AB_1)^2P^2-(BC_1-B_1C)^2P$,
so it divides $P\{(B^2C_1^2+B_1^2C^2)-(A_1^2B^2+A^2B_1^2)P\}
=2B^2B_1^2PQ$.
\end{proof}

\begin{rem}
The curve $x=\infty$ is an exceptional curve if and only if $\deg (EP) > \deg D$.
\end{rem}

\begin{prop}
Let $T$ be the map as in $(\ref{T})$.
\hfill\break
$(1)$
When $P(c)=0$ and $Q(c) \not= 0$, $x=c$ is an exceptional curve of $T$
if and only if
$\frac{C}{B}=-\frac{C_1}{B_1}$ at $x=c$.
\hfill\break
$(2)$
When $P(c)=Q(c)=0, Q(c) \not= 0$, $x=c$ is an exceptional curve of $T$
if and only if
$\frac{A}{B}=-\frac{A_1}{B_1}$ at $x=c$.

Moreover, $x=c$ is mapped biregularly to the blowing up $E_c$
of the point $(c,0)$.
\end{prop}

\begin{proof}
$P(c)=0$ implies $B(c) \not=0, B_1(c) \not=0$ and
$Q=\left(\frac{C}{B}\right)^2=\left(\frac{C_1}{B_1}\right)^2$ at $x=c$,
so $\frac{C}{B}=\pm\frac{C_1}{B_1}$ at $x=c$.
If $Q(c) \not= 0$, we have $D(c)=0$
if and only if $\frac{C}{B}=-\frac{C_1}{B_1}$ at $x=c$,
hence we get (1).
If $P(c)=Q(c)=0$,
none of $A,A_1,B,B_1$ is zero at $x=c$
but $C(c)=C_1(c)=0$,
thus we get $\left(\frac{A}{B}\right)^2=
\left(\frac{A_1}{B_1}\right)^2=-\frac{Q}{P}$ at $x=c$,
and similar discussion as above leads to the result (2).

In both cases, $c$ is a simple zero of $D^2-E^2P$,
since $E(c) \not=0$ and $c$ is a simple zero of $P$.
$T$ maps $x=c$ to a point $(c,0)$,
so $T$ maps $x=c$ biregularly to the blowing up $E_c$ at $(c,0)$ biregularly.
\end{proof}

\begin{prop} \label{prop6}
Let $T$ be the map as in $(\ref{T})$.
When $Q(c)=0$ and $P(c) \not= 0$,
$x=c$ is an exceptional curve of $T $if and only if
$\frac{C}{A}=-\frac{C_1}{A_1}$ at $x=c$.
Moreover, $x=c$ is mapped to the blowing up $E_c$ at the point
$(c,-C(c)/A(c))$.
\end{prop}

\begin{proof}
$Q(c)=0$ implies $A(c)\not=0, A_1(c) \not=0$
and $P=\left(\frac{C}{A}\right)^2=\left(\frac{C_1}{A_1}\right)^2$
at $x=c$,
so $\frac{C}{A}=\pm\frac{C_1}{A_1}$ at $x=c$.
$D^2-E^2P=0$ is equivalent with
$\left(\frac{C}{A}\right)^2=\left(\frac{D}{E}\right)^2$.
We see that if $\frac{C}{A}=-\frac{C_1}{A_1}$,
then $\frac{D}{E}=-\frac{C}{A}$,
but if $\frac{C}{A}=\frac{C_1}{A_1}$,
then $\frac{D}{E} \not= \pm \frac{C}{A}$.
This leads to the first statement of Proposition \ref{prop6}.

Since $D^2-E^2P$ divides $B^2B_1^2PQ$ and $BB_1$ is disjoint with $Q$,
$c$ is a simple zero of $D^2-E^2P$.
Since $T$ maps $x=c$ to a point $(c,-C(c)/A(c))$,
this leads to the second statement of Proposition \ref{prop6}.
\end{proof}

\begin{prop}
Let $T$ be the map as in $(\ref{T})$ and $A, B, C,A_1, B_1, C_1$ be as in Proposition \ref{prop44}.
When $c$ is a zero of $BB_1$,
 $x=c$ is an exceptional curve of $T$
 if and only if
 \hfill\break
$(1)$
 only one of $B(c)$ or $B_1(c)$ is zero, or
 \hfill\break
$(2)$
 $B(c)=B_1(c)=0$ and $\frac{C}{A}=-\frac{C_1}{A_1}$ at $x=c$.

In case $(1)$,
$x=c$ is mapped to the blowing up at
$(c, -C(c)/A(c))$ or $(c,C_1(c)/A_1(c))$.
In case $(2)$,
$x=c$ is mapped to the blowing up of order 2
at $(c, -C(c)/A(c))$.
\end{prop}

\begin{proof}
If $B(c)=0$ and $B_1(c) \not=0$,
then neither of $A$ nor $C$ is zero and $P=\left(\frac{C}{A}\right)^2$ at $x=c$.
On the other hand, $\frac{D}{E}=\frac{B_1C+BC_1}{A_1B-AB_1}
=-\frac{C}{A}$ at $x=c$.
From this, $c$ is a simple zero of $D^2-E^2P$
(simplicity comes from that $c$ is a simple zero of $B$)
and we get (1).

If $B(c) \not= 0$ and $B_1(c)=0$,
similar discussions hold,
replacing $-C/A$ by $C_1/A_1$.

If $B(c)=B_1(c)=0$,
we have $P=\left(\frac{C}{A}\right)^2=\left(\frac{C_1}{A_1}\right)^2$
at $x=c$,
so $\frac{C}{A}=\pm\frac{C_1}{A_1}$ at $x=c$.
When $\frac{C}{A}=-\frac{C_1}{A_1}$
we have $\frac{D}{E}=-\frac{C}{A}$ at $x=c$,
and $c$ is a double zero of $D^2-E^2P$.
Thus $x=c$ is an exceptional curve of $T$
and $T$ maps $x=c$ to the blowing up of order 2 at
$(c,-C(c)/A(c))$.
The fundamental point of $T$ is $(c, C(c)/A(c) )$.
Then $T$ maps $E_c^\prime$,
the blowing-up at this point,
to $E_c$, the blowing up at the image point.
When $\frac{C}{A}=\frac{C_1}{A_1}$ at $x=c$,
we have $\frac{D}{E} \not= \pm \frac{C}{A}$ at $x=c$,
and $x=c$ is not an exceptional curve of $T$.
\end{proof}

The curve $x=\infty$ may or may not be an exceptional curve of $T$,
but we have the following:

\begin{prop} \label{prop8}
Let $T$ be the map as in $(\ref{T})$ and $F_\infty$ be the blowing up of order $\lfloor d_P/2 \rfloor
=\max\{ m \in \mathbb{Z} \, | \, m \leq d_P/2 \}$ at $(\infty,\infty) \in \mathbb{P}^1 \times \mathbb{P}^1$.
Then $T$ maps $F_\infty$ to $F_\infty$ except the following three cases:
\hfill\break
$(1)$
$d_P$ is even, $d_Q$ is odd and
$\frac{a_0}{c_0}=-\frac{a_{1,0}}{c_{1,0}}$;
\hfill\break
$(2)$
$d_P$ is odd, $d_Q$ is even and
$c_0=-c_{1,0}$;
\hfill\break
$(3)$
$d_P$ is odd, $d_Q$ is odd and
$a_0=-a_{1,0}$.

Here, $a_0$ (resp. $c_0$, $a_{1,0}$, $c_{1,0}$) is the coefficient of the highest degree term
of $A$ (resp. $C$, $A_1$, $C_1$).
In these three cases,
$T$ maps $F_\infty$ to $E_\infty$,
once more blowing up of $F_\infty$.
\end{prop}

\begin{proof}
Let $r=\deg D-\deg E$.
By checking the order of infinity at $x=\infty$ of $u$ and $u_1=\frac{Du+EP}{Eu+D}$,
we see that when $d_P$ is odd,
$F_\infty$ is mapped to $F_\infty$
if $r > \frac{d_P}{2}$
and to $E_\infty$ if $r < \frac{d_P}{2}$.

When $d_Q$ is even,
$d_C-d_A>d_P/2$ and $c_0^2=c_{10}^2=q_0$.
From this we see that
$r<\frac{d_P}{2}$
if and only if $c_0=-c_{10}$.
When $d_Q$ is odd,
$d_C-d_A<d_P/2$ and $a_0^2=a_{10}^2=-q_0/p_0$.
From this we see that $r < \frac{d_P}{2}$ if and only if $a_0=-a_{10}$.
Note that $a_0^2=a_{10}^2$ always.

When $d_P$ is even,
$T$ maps $F_\infty$ to $F_\infty$ whenever $r \not= \frac{d_P}{2}$.
If $r=\frac{d_P}{2}$,
$u \sim \lambda x^r$ implies
$u_1 \sim \frac{\lambda d_0+e_0p_0}{\lambda e_0+d_0} x^r$,
so that $T$ maps $F_\infty$ to $F_\infty$
unless $d_0^2-e_0^2p_0=0$.
But we can verify that
$r=\frac{d_P}{2}$ and $d_0^2-e_0^2p_0=0$
are equivalent with that
$d_Q$ is odd and $\frac{a_0}{c_0}=-\frac{a_{10}}{c_{10}}$.
In this case, we need once more blowing up,
and $T$ maps $F_\infty$ to $E_\infty$.
\end{proof}


\subsection{Construction of $Y$.} \label{pre3}

For a given $A, B, C$,
$u=\frac{By+A}{Bz+C}$ is mapped to $\frac{B^\sigma y+A^\sigma}{B^\sigma z+C^\sigma}$
by the action of $\sigma \in \mathfrak{G}=\Gal(l/k)$.
Here $A^\sigma$ is the polynomial obtained from $A$
by replacing all coefficients to its conjugates by $\sigma$.

We shall construct a non-singular projective surface $Y$ on which
$\mathfrak{G}$ acts in a Zariski homeomorphic way.
Starting from $\mathbb{P}^1 \times \mathbb{P}^1$,
we repeat blowings-up and blowings-down.

The automorphism $\Phi_\sigma$ of $l(x,u)$
is induced by the point transformation $\Psi_\sigma$ of
$\mathbb{P}^1 \times \mathbb{P}^1$
as $(\Phi_\sigma f)(x,y)=\Big(f\big(\Psi_\sigma^{-1}(x,y)\big)\Big)^\sigma$.
Here $\Psi_\sigma=\tau_\sigma \circ \widetilde \sigma$,
$\widetilde \sigma: (x,u) \mapsto (x^\sigma, u^\sigma)$
and $\tau_\sigma: (x,u) \mapsto
(x, \frac{D_\sigma u-E_\sigma P}{-E_\sigma u+D_\sigma})$,
where
\begin{equation}
E_\sigma \text{ and } D_\sigma \text{ are mutually disjoint and }
\frac{D_\sigma}{E_\sigma}=\frac{BC^\sigma+B^\sigma C}{A^\sigma B-AB^\sigma}
=\frac{(A^\sigma B+AB^\sigma)P}{BC^\sigma-B^\sigma C}.
\end{equation}

\begin{prop} \label{prop9}
The Galois group $\mathfrak{G}$ acts on some $Y_{rs}$
in a Zariski homeomorphic way.
($Y_{rs}$ is defined in Subsection \ref{pre2}, and $r$ and $s$ are given in the proof.)
\end{prop}

\begin{proof}
$\Psi_\sigma=\tau_\sigma \circ \widetilde \sigma$ is Zariski homeomorphic
except on the line $x=c$,
where $c$ is a zero of $P$ or a zero of $Q$ or a conjugate of a zero of $B$ or $c=\infty$.

Let $s=s_1+s_2+s_3+s_4$
where $s_1$ (resp. $s_2$, resp. $s_3$) is the number of
$c \in k^{\sep}$
such that $P(c)=0$ and
$Q(c) \not\in k(c)^2$
(resp. $Q(c)=0$ and $P(c) \not\in k(c)^2$,
resp. $P(c)=Q(c)=0$ and $-\frac{Q}{P}(c) \not\in k(c)^2$),
and let $E_c$ be the blowing up at $(c,0)$
(resp. $(c, C(c)/A(c)$), resp. $(c,0)$).
One more $E$ at $(\infty,\infty)$ is added for three cases
(i), (ii), (iii) stated later.
So, $s_4=1$ for these cases and $0$ otherwise.

Suppose that $P(c)=0$ and $Q(c) \not\in k(c)^2$.
Let $H_c=\{ \sigma \in \mathfrak{G} | c^\sigma=c\}$.
Since $Q(c)=\big(\frac{C(c)}{B(c)}\big)^2 \in k(c) \setminus k(c)^2$,
we see that $\frac{C}{B}(c)=\frac{C^\sigma}{B^\sigma}(c)$
for half ones of $\sigma \in H_c$
and $\frac{C}{B}(c)=-\frac{C^\sigma}{B^\sigma}(c)$
for other half ones.

From this, we see that for any conjugates $c^\prime$ and $c^{\prime\prime}$ of $c$,
half ones of $\Psi_\sigma$ such that $c^{\prime\sigma}=c^{\prime\prime}$ map
$x=c^\prime$ to $x=c^{\prime\prime}$, $E_{c^\prime}$ to $E_{c^{\prime\prime}}$
and other half ones of $\Psi_\sigma$ map
$x=c^\prime$ to $E_{c^{\prime\prime}}$, $E_{c^\prime}$ to $x=c^{\prime\prime}$.
Note that on the blown up surface,
the class of $x=c^\prime$ is $F-E_{c^\prime}$.
So the number of $\Psi_\sigma$ which map
$E_{c^\prime}$ to $E_{c^{\prime\prime}}$
is the same with the number of $\Psi_\sigma$
which map $E_{c^\prime}$ to $F-E_{c^{\prime\prime}}$.

Let $r=r_1+r_2+r_3+r_4+r_5$.
Here $r_5=\deg P/2$ or $(\deg P-1)/2$ according to whether $\deg P$ is even or odd.
$r_4=\deg B$ and $r_1,r_2,r_3$ are given below.

Suppose that $P(c)=0, Q(c) \not=0, Q(c) \in k(c)^2$.
Then $Q(c)=\big(\frac{C(c)}{B(c))}\big)^2 \in k(c)^2$
implies $\frac{C}{B}(c)=\frac{C^\sigma}{B^\sigma}(c)$
for all $\sigma \in H_c$.
From this, we see that for any conjugates of $c^\prime$ and $c^{\prime\prime}$
of $c$, all $\Psi_\sigma$ such that $c^{\prime\sigma}=c^{\prime\prime}$
map $x=c^\prime$ to $x=c^{\prime\prime}$,
$E_{c^\prime}$ to $E_{c^{\prime\prime}}$
or all $\Psi_\sigma$ map $x=c^\prime$ to $E_{c^{\prime\prime}}$,
$E_{c^\prime}$ to $x=c^{\prime\prime}$.
So, we can divide such $c$ into two blocks,
so that if $c^\prime$ and $c^{\prime\prime}$ are in the same block,
all $\Psi_\sigma$ map $x=c^\prime$ to $x=c^{\prime\prime}$,
$E_{c^\prime}$ to $E_{c^{\prime\prime}}$
and if $c^{\prime\prime}$ belongs to the different block with $c^\prime$,
all $\Psi_\sigma$ map $x=c^\prime$ to $E_{c^{\prime\prime}}$,
$E_{c^\prime}$ to $x=c^{\prime\prime}$.

Let $r_1$ be the number of $c$ in the second block.
When $c$ is in the first block,
let $F_c$ be $(x=c)$.
When $c$ is in the second block,
let $F_c$ be the blowing down of $E_c$ by $(x=c)$.
Then $\Psi_\sigma$ maps always $F_{c^\prime}$ to $F_{c^{\prime\prime}}$.
Note that on the blown-up and down surface,
the class of $F_c$ is $F$.
So $\Psi_\sigma$ induces the transformation in the same class $F$.

The same discussions hold for $Q(c)=0, P(c) \not=0, P(c) \in k(c)^2$
and also for $P(c)=Q(c)=0, -\frac{Q}{P}(c) \in k(c)^2$.

When $B(c)=0$, (then $P(c) \not= 0$ and $Q(c) \not= 0$),
let $F_c$ be the blowing up at $(c, C(c)/A(c))$,
blown down by $(x=c)$ afterwards.
When $c$ is a conjugate of a zero of $B$ and $B(c) \not= 0$,
let $F_c$ be $(x=c)$.
Then $\Psi_\sigma$ map $F_{c^\prime}$ to $F_{c^{\prime\prime}}$ always.
Thus $\Psi_\sigma$ induces an automorphism of $\Pic$
which keeps the class $F$ unchanged.

For $x=\infty$, let $F_\infty$ be the blowing up of order $\deg P/2$
or $(\deg P-1)/2$ at $(\infty,\infty)$,
blown down by $E_j$ (=blowings up of smaller orders) afterwards.
(See the discussions in Subsection \ref{biregT}.)
Then the class of $F_\infty$ is $F$.
$\Psi_\sigma$ maps $F_\infty$ to $F_\infty$
except the following three cases:
\begin{sub}
\item[(i)] $\deg P$ even, $\deg Q$ odd and $p_0 \not\in k^2$;
\item[(ii)] $\deg P$ odd, $\deg Q$ even and $q_0 \not\in k^2$;
\item[(iii)] $\deg P$ odd, $\deg Q$ odd and $-q_0/p_0 \not\in k^2$.
\end{sub}
In these cases, take once more blowing-up $E_\infty$ of $F_\infty$,
then half ones of $\sigma \in \mathfrak{G}$ map
$E_\infty $ to $E_\infty$ and other half ones map $E_\infty$ to $F_\infty$.

The results are summarized as follows.
The blowings up are only $E_c$ such that $P(c)=0, Q(c) \not\in k(c)^2$,
or $Q(c)=0, P(c) \not\in k(c)^2$
or $P(c)=Q(c)=0, -\frac{P}{Q}(c) \not\in k(c)^2$.
The blowings up and down afterwards are at other zeros of $BPQ$.
(For the zeros of $PQ$, only at $c$ in  the second block.)

For $x=\infty$, the blowings up and down $F_\infty$ is added,
and the blowing up $E_\infty$ is added in some cases
((i), (ii), (iii) mentioned above).

Thus we reach the desired surface $Y_{rs}$.
\end{proof}

\subsection{$\Pic(Y)$ and $\Pic(Y)^\mathfrak{G}$.}

From the discussion in the previous subsection,
we can determine $\Pic(Y)$ and $\Pic(Y)^\mathfrak{G}$ as follows.

$\Pic(Y)$ is of rank $s+2$ as a $\mathbb{Z}$-module
with the basis $E_i \, (1 \leq i \leq s)$, $F$ and  $F^\prime$.
The action of $\sigma \in \mathfrak{G}$ is represented
as the following matrix with the basis above in this order:
\begin{equation} \label{matg}
g_\sigma=
\begin{pmatrix}
A_\sigma & X_\sigma & 0 \\
0 & 1 & 0 \\
U_\sigma & \alpha_\sigma & 1
\end{pmatrix}.
\end{equation}
Here $A_\sigma$ is an $s \times s$ matrix
whose entries are $0$, $1$ or $-1$.
The matrix obtained by replacing all the entries $-1$'s of $A_\sigma$ by 1
is the permutation matrix
which represents the permutation of $c_i$ induced by $\sigma$.
For any $(i,j)$, $1 \leq i,j \leq s$,
the number of $A_\sigma$ where $(i,j)$-entry is $1$
is equal with the number of $A_\sigma$ whose $(i,j)$-entry is $-1$.

$X_\sigma$ is a column vector whose entries are $0$ or $1$.
The $j$-th entry is $0$ or $1$ according to whether the only one
non-zero entry of $A_\sigma$ in $j$-th row is $1$ or $-1$.

$U_\sigma$ is a row vector whose entries are $-1$ or $0$.
The $i$-th entry is $-1$ or $0$ according to whether
$(x=c_i)$ is an exceptional curve of $T_\sigma$ or not,
so according to whether the only on non-zero entry of $A_\sigma$
in $i$-th column is $-1$ or $1$.

$\alpha_\sigma$ is some integer.
Its value is determined in Subsection \ref{s1}.
From $(\ref{matg})$,
$\sum_\sigma g_\sigma$ is calculated as follows:
\begin{equation}
\sum_\sigma g_\sigma=
\frac{|\mathfrak{G}|}{2}
\begin{pmatrix}
O_s & 1 & 0 \\
0 & 2 & 0 \\
-1 & \alpha & 2
\end{pmatrix}
\end{equation}
where $O_s$ is the $s \times s$ matrix whose all entries are zero.
$\frac{|\mathfrak{G}\,|}{2}\alpha=\sum_\sigma \alpha_\sigma$,
so $\alpha$ may not be an integer.

This matrix is the same with $(\ref{matsg})$ except the $(s+2,s+1)$ entry.
(But $r$ in $(\ref{matsg})$ is replaced by $s$ here.)
So, Proposition \ref{prop10} holds.
For the sake of convenience, we shall write it again.

\begin{prop} \label{prop412}
The submodule $\Pic(Y)^\mathfrak{G}$ of $\mathfrak{G}$-invariant classes
has rank 2 with the basis $F$ and $\Omega$.
We have $F \cdot F=0, \, F \cdot \Omega=-2$
and $\Omega \cdot \Omega=8-s$.
\end{prop}

\begin{rem}
Proposition \ref{prop412} holds for the case $s>0$.
For the case $s=0$, $\Pic(Y)$ is of rank 2
and every class is $\mathfrak{G}$-invariant.
\end{rem}

\subsection{Irrationality for $s \geq 6$ and $s=4$.}

The discussions in Subsection \ref{Ginv} to Subsection \ref{updown2}
rely only on the structure of $\Pic(Y)^\mathfrak{G}$.
Since $\Pic(Y)^\mathfrak{G}$ is the same,
the discussions there can be applied to prove
the irrationality of $k(x,y,z)$
(but $r$ in Subsection \ref{Ginv} to Subsection \ref{updown2} should be replaced by $s$ here).

By the discussions in Subsection \ref{Ginv} and Subsection \ref{pfB},
$k(x,y,z)$ is not $k$-rational when $s \geq 8$.
By the discussions in Subsection \ref{deg46} and Subsection \ref{updown2},
$k(x,y,z)$ is not $k$-rational when $s=6$ or $4$.
This argument can be applied also for $s=7$,
putting $\omega=1$ instead of $\omega=2$ or $4$.
See Theorem \ref{lem2}.

Only remained are the proof of irrationality for $s=5$
and the proof of rationality for $s \leq 3$.

\subsection{The case $s=5$.}

Assume that $k(x,y,z)$ is $k$-rational,
and we shall derive a contradiction.
Let $Y^\prime$ be the del Pezzo surface obtained from $Y$
(see Subsection \ref{deg46}).
On the surface $Y^\prime$, the class $D=-F-\Omega_{Y^\prime}$ satisfies
$D \cdot D=D \cdot \Omega_{Y^\prime}=-1$.
Similarly as Theorem \ref{lem2} and Theorem \ref{lem3},
we can verify that there exists a unique irreducible curve $L$ which belongs to $D$.
Let $Y^{\prime\prime}$ be the blowing down of $Y^\prime$ by $L$.
Then, both of $F$ and $-\Omega_{Y^\prime}$ in $\Pic(Y^\prime)$ are mapped to
$-\Omega_{Y^{\prime\prime}}$ in $\Pic(Y^{\prime\prime})$.
Thus, $\Pic(Y^{\prime\prime})^\mathfrak{G}$ has rank 1 with $\Omega_{Y^{\prime\prime}}$ as its basis.

We shall write $Y^{\prime\prime}$ as $Y$.
The problem is rewritten on the surface $Y$ as follows.
$Y$ is a del Pezzo surface with $\omega=4$.
$\mathfrak{G}$ acts Zariski homeomorphic on $Y$
and $\Pic(Y)^\mathfrak{G}$ has rank 1 with the basis $\Omega$.

If $k(x,y,z)$ is $k$-rational,
there exist $\mathfrak{G}$-invariant irreducible curves $C, C^\prime$ on $Y$
such that we can reach $C \cdot C=0, C \cdot \Omega=-2, C \cdot C^\prime=1$
after some blowings-up $\{E_j^\prime\}$.

Suppose $C \equiv -m\Omega$ in $\Pic(Y)$,
then $C \cdot C=4m^2$ and $C \cdot \Omega=-4m$.
So we have
$$\sum_j m_j^2=4m^2,
\sum_j m_j=4m-2$$
where $m_j=C \cdot E_j^\prime$.
Especially we have $m>0$.

\begin{prop}
Put $m_1=\max_j m_j$.
Then
\hfill\break
$(1)$
$m_1>m$,
\hfill\break
$(2)$
the number of $j$ such that $m_j=m_1$
is at most $3$.
\end{prop}

\begin{proof}
\hfill\break
$(1)$
It suffices to show that $\sum m_j^2 \big/ \sum m_j>m$,
namely $4m^2 \big/(4m-2) >m$.
However,
$\frac{4m^2}{4m-2}=m\frac{4m}{4m-2}>m$ is evident.
\hfill\break
$(2)$
Suppose that the number of the desired $j$ is $q$, then we have
$qm_1 \leq \sum m_j=4m-2$,
so that $q \leq \frac{4m-2}{m_1} < \frac{4m_1-2}{m_1} < 4$.
\end{proof}

Since the family of the blowings-up $\Phi$ commutes with the action of $\mathfrak{G}$,
the configuration of $\{E_j^\prime\}$ is $\mathfrak{G}$-invariant,
and the action of $\mathfrak{G}$ induces a permutation of $j$
with $m_j=m_{\sigma(j)}$ since $C$ is $\mathfrak{G}$-invariant.
This implies that the base point $P_1$ of $E_1^\prime$
has at most three conjugates including $P_1$ itself.

\begin{prop}
When $s=5$, $k(x,y,z)$ is not $k$-rational.
\end{prop}

\begin{proof}
We split the proof into three cases:
\hfill\break
$(\mathrm{I})$
When $P_1$ is a $\mathfrak{G}$-invariant point.

Let $\widetilde Y$ be the blowing-up of $Y$ at $P$.
We have $\Omega_{\widetilde Y} \cdot \Omega_{\widetilde Y}=3$.
We shall show that $\widetilde Y$ is a del Pezzo surface.
Let $\Gamma$ be an irreducible curve on $\widetilde Y$.
If $\Gamma=E_1^\prime$,
then $\Gamma \cdot \Omega_{\widetilde Y}=-1$.
Other $\Gamma$ comes from an irreducible curve on $Y$,
and $\Gamma \cdot \Omega_{\widetilde Y}=\Gamma \cdot \Omega_Y+m_1^\prime$
where $m_1^\prime=\Gamma \cdot E_1^\prime$.
Suppose that $\Gamma \cdot \Omega_Y=-\alpha, \alpha>0$,
then $\Gamma \cdot C=m\alpha$ and $C \cdot E_1^\prime=m_1$
imply that $m\alpha \geq m_1m_1^\prime$,
so that $m_1^\prime \leq \frac{m\alpha}{m_1}<\alpha$,
which means $\Gamma \cdot \Omega_{\widetilde Y}<0$.
Thus $\widetilde Y$ is a del Pezzo surface.

$\Pic(\widetilde Y)^\mathfrak{G}$ has rank 2 with $\Omega_{\widetilde Y}$ and $E_1^\prime$
as its basis.
Put $F=-E_1^\prime-\Omega_{\widetilde Y}$,
then $F \cdot F=0, F \cdot \Omega_{\widetilde Y}=-2$ and $\Pic(\widetilde Y)^\mathfrak{G}$ is
generated by $F$ and $\Omega_{\widetilde Y}$.
Since $\Omega_{\widetilde Y}=\Omega_Y+E_1^\prime$,
we have $C \equiv -m\Omega_{\widetilde Y}-(m_1-m)E_1^\prime
=(m_1-m)F-(2m-m_1)\Omega_{\widetilde Y}$.
If $2m-m_1=0$, then $C \cdot C=0$ on $\widetilde Y$
and $C \cdot C^\prime=1$ can not happen for other
$\mathfrak{G}$-invariant irreducible curve $C^\prime$
as mentioned in Subsection \ref{updown2}.
Otherwise $2m-m_1>0$.
For all $j \geq 2$
such that $m_j>2m-m_1$, after blowing-up $E_j^\prime$,
blow down by the irreducible curve which belongs to $F$
and passes through the base point of $E_j^\prime$.
On the obtained surface $\overline{\widetilde Y}$,
we have $C \equiv \nu F-(2m-m_1)\Omega$ with $\nu \leq -1$.

Let $Z$ be the blow down of $\overline{\widetilde Y}$ by the irreducible curve
belonging to $-F-\Omega$.
Then $Z$ is a del Pezzo surface with $\omega=4$,
and we have $C \equiv -(2m-m_1+\nu)\Omega$ in $\Pic(Z)$.
This is the same situation with the original $Y$,
but $m$ is replaced by $\mu=2m-m_1+\nu <m+\nu <m$.
\hfill\break
$(\mathrm{II})$
When $P_1$ and $P_2$ are mutually conjugate.

Let $\widetilde{\widetilde Y}$ be the blowing-up of $Y$ at $P_1$ and $P_2$.
We have $\Omega_{\widetilde{\widetilde Y}} \cdot \Omega_{\widetilde{\widetilde Y}}=2$,
and $\widetilde{\widetilde Y}$ is a del Pezzo surface by the same reason as $(\mathrm{I})$.
$\mathfrak{G}$ acts Zariski homeomorphic on $\widetilde{\widetilde Y}$
and $\Pic(\widetilde{\widetilde Y})^\mathfrak{G}$ has rank 2
with $\Omega_{\widehat{\widehat Y}}$ and $E_1^\prime+E_2^\prime$
as its basis.
Since $\Omega_{\widetilde{\widetilde Y}}=\Omega_Y+E_1^\prime+E_2^\prime$,
we have
$$C \equiv -m\Omega_{\widetilde{\widetilde Y}}-(m_1-m)(E_1^\prime+E_2^\prime).$$
By Theorem \ref{lem3}, for $i=1,2$, there exists a unique irreducible curve $L_i$ such that
$L_i \equiv -E_i^\prime-\Omega_{\widetilde{\widetilde Y}}$
and $L_i \cdot L_i=L_i \cdot \Omega_{\widetilde{\widetilde Y}}=-1$.
Let $Z$ be the blowing down of $\widetilde{\widetilde Y}$ by $L_1$ and $L_2$.
Since all of $E_1^\prime, E_2^\prime$ and $-\Omega_{\widetilde{\widetilde Y}}$
in $\Pic(\widetilde{\widetilde Y})$ are mapped to $-\Omega_Z$ in $\Pic(Z)$,
we have
$C \equiv -(3m-2m_1)\Omega_Z$ in $\Pic(Z)$.
This is the same situation with the original $Y$,
but $m$ is replaced by $\mu=3m-2m_1<m$.
\hfill\break
$(\mathrm{III})$
When $P_1, P_2$ and $P_3$ are mutually conjugate.

Let $\widetilde{\widetilde{\widetilde Y}}$ be the blowing-up of $Y$
at $P_1, P_2$ and $P_3$.
We have $\Omega \cdot \Omega=1$ on $\widetilde{\widetilde{\widetilde Y}}$,
and $\widetilde{\widetilde{\widetilde Y}}$ is a del Pezzo surface by the same reason as $(\mathrm{I})$.
$\mathfrak{G}$ acts Zariski homeomorphic on $\widetilde{\widetilde{\widetilde Y}}$
and $\Pic(\widetilde{\widetilde{\widetilde Y}})^\mathfrak{G}$ has rank 2
with $\Omega$ and $E_1^\prime+E_2^\prime+E_3^\prime$ as its basis.

By Theorem \ref{lem3}, for $i=1,2,3$, there exists a unique irreducible curve $L_i$
such that $L_i \equiv -E_i^\prime-2\Omega$ and
$L_i \cdot L_i=L_i \cdot \Omega=-1$.
Let $Z$ be the blowing-down of $\widetilde{\widetilde{\widetilde Y}}$ by
$L_1$, $L_2$ and $L_3$.
Then we have
$C \equiv -(7m-6m_1)\Omega_Z$ in $\Pic(Z)$,
and $\mu=7m-6m_1<m$.

In any cases of $(\mathrm{I})$, $(\mathrm{II})$ and $(\mathrm{III})$, we can replace $Y$ with $Z$ with a smaller value of $m$.
Repeat this procedure.
After finite steps, we reach $m \leq 0$,
which means that $C \cdot C=0, C \cdot \Omega=-2$
can not be reached by any blowings-up $\{E_j^\prime\}$.
\end{proof}

\subsection{Impossibility of $s=1$.} \label{s1}

Consider the action of $\mathfrak{G}$ on $\Pic(Y)$.

The image of $\Psi_\sigma-\id$ should be contained in the kernel of $\sum_\sigma \Psi_\sigma$.
$\Psi_\sigma-\id$ maps $F^\prime$ to $-\sum^\prime E_i+\alpha_\sigma F$
for some integer $\alpha_\sigma$,
the sum $\sum^\prime$ being taken over such $i$ that
$(x=c_i)$ is an exceptional curve of $\tau_\sigma$.

$\sum \Psi_\sigma$ maps $E_i$ to
$\frac{|\mathfrak{G}|}{2}F (1 \leq i \leq s)$,
and $F$ to $|\mathfrak{G}|F$,
so $-\sum^\prime E_i+\alpha_\sigma F$ to
$\frac{|\mathfrak{G}|}{2}(-n_\sigma+2\alpha_\sigma)F$
where $n_\sigma$ is the number of $i$
such that $(x=c_i)$ is an exceptional curve of $\tau_\sigma$.

From this we see that $n_\sigma=2\alpha_\sigma$,
so $n_\sigma$ is even.

Suppose that $s \not= 0$,
and $x=c_1$ is an exceptional curve of $\tau_\sigma$.
Then, since $n_\sigma \geq 2$,
there exists at least one more exceptional curve,
so that $s \geq 2$.
This shows that $s \not=0$ implies $s \geq 2$,
so the case $s=1$ never happens.

\subsection{The case where $s=0$.} \label{s0}

The action $\tau_\sigma \circ \widetilde \sigma$ is Zariski homeomorphic on $Y$,
which is obtained from $\mathbb{P}^1 \times \mathbb{P}^1$
by $r$-times blowing up and down.
$\Pic(Y)$ is of rank 2 with the basis $F, F^\prime$.
We have $F \cdot F=0, F \cdot F^\prime=1, F^\prime \cdot F^\prime=r$.
Every class is $\mathfrak{G}$-invariant.

\begin{prop} \label{prop415}
If there exists a $\mathfrak{G}$-invariant curve
linear in $u$, namely $G_1(x)u+G_2(x)=0$,
then $l(x,u)=l(x,w)$
with some $\mathfrak{G}$-invariant $w$.
\end{prop}

\begin{proof}
Put $v=G_1(x)u+G_2(x)$,
then $(x,v)$ is a transcendent basis of $l(x,u)$,
and the curve $v=0$ is $\mathfrak{G}$-invariant.
Since $\tau_\sigma$ keeps $x$ invariant,
$\tau_\sigma$ maps $v$ to a linear fraction of $v$
with $l(x)$-coefficients, namely
$$\tau_\sigma:
v \mapsto \frac{\alpha_\sigma(x)v+\beta_\sigma(x)}{\gamma_\sigma(x)v+\delta_\sigma(x)},
\alpha_\sigma, \beta_\sigma, \gamma_\sigma, \delta_\sigma \in l[x].$$
But since $v=0$ is invariant,
we have $\beta_\sigma=0$ for any $\sigma$.
Put $v^\prime=\frac{1}{v}$, then
$$\tau_\sigma: v^\prime \mapsto
\frac{\delta_\sigma(x)v^\prime+\gamma_\sigma(x)}{\alpha_\sigma(x)}.$$
Namely $\tau_\alpha: v^\prime \mapsto
\alpha_\sigma(x)v^\prime+\beta_\sigma(x),
\alpha_\sigma, \beta_\sigma \in l(x)$.
From this we see that some $w=\alpha(x)v^\prime+\beta(x)$
($\alpha, \beta \in l(x), \alpha \not=0$)
is $\mathfrak{G}$-invariant (see \cite{HK95}).
Thus we obtained a $\mathfrak{G}$-invariant transcendent basis $(x,w)$.
\end{proof}

\begin{prop} \label{prop16}
When $s=0$, $k(x,y,z)$ is $k$-rational
except the following case.
Both of $\deg P$ and $\deg Q$ are even and $a^2p_0+b^2q_0=c^2$
has no non-zero solution $(a,b,c)$ in $k$.
\end{prop}

\begin{proof}
Assume that some class
$\nu F+F^\prime$ with $\nu<-\frac{r}{2}$
contains an irreducible curve $\Gamma$,
then $\Gamma \cdot \Gamma=2\nu+r<0$,
so that $\Gamma$ is $\mathfrak{G}$-invariant
because $\Gamma$ is the unique irreducible curve in this class.
Therefore $k(x,y,z)$ is $k$-rational by Proposition \ref{prop415}.

When $r$ is odd, such $\nu$ exists.
Let $d=\frac{r-1}{2}$. Then $G_1(x)u+G_2(x)$ with
\break
$\deg G_1, \deg G_2 \leq d$
has $2(d+1)=r+1$ coefficients.
The condition that $G_1(x)u+G_2(x)=0$ passes through $r$ points
yields $r$ linear equations on these coefficients,
so there exists a non-zero solution.
Suppose that the curve passes through $r$ blown up points,
then $\nu=d-r=\frac{r-1}{2}-r=-\frac{r+1}{2}<-\frac{r}{2}$.
Since $\deg B$ is sufficiently large by Proposition \ref{prop44} (2),
we have $\deg B>d$,
so neither $G_1$ nor $G_2$ is zero.
If $G_1$ and $G_2$ are not mutually disjoint,
divide by GCD to get an irreducible curve.
Thus when $r$ is odd,
$k(x,y,z)$ is $k$-rational.

Suppose that $r$ is even and let $d=\frac{r}{2}$.
Similar discussion as above shows that the curves
$G_1(x)u+G_2(x)=0$ with $\deg G_1, \deg G_2 \leq d$
which pass through all $r$ blown up points
form a vector space of at least two dimensional.

If the dimension is 3 or more,
there exists a solution among them such that the coefficients
of the highest degree terms of $G_1$ and $G_2$ are zero,
and for such a solution, we have $\nu<d-r=\frac{r}{2}-r=-\frac{r}{2}$,
so the problem is reduced to the solved case.

Suppose that the desired solutions are two dimensional,
then the family of desired curves is parametrized by $\mathbb{P}^1$.
They are mutually disjoint in $Y$ because of $\Gamma \cdot \Gamma=0$,
and their union covers all $Y$.
The relation that $\Gamma$ passes through $P$ defines
a one-to-one correspondence between a point $P$ on $F_\infty$
and a curve $\Gamma$ in this family,
because of $\Gamma \cdot F_\infty=1$.

Consider the action of $\Psi_\sigma$ on $F_\infty$.
If there exists a $\mathfrak{G}$-invariant point on $F_\infty$,
then the corresponding curve is $\mathfrak{G}$-invariant,
so $l(x,u)$ has a transcendent basis $(x,w)$
with some $\mathfrak{G}$-invariant $w$.

On the contrary, if $\mathfrak{G}$-invariant point does not exist
on $F_\infty$, then $\mathfrak{G}$-invariant $\Gamma$ does not exist,
so $\mathfrak{G}$-invariant point does not exist on $Y$ at all,
because the curves are mutually disjoint.
Hence $l(x,u)$ can never have a $\mathfrak{G}$-invariant
transcendent basis.
Thus the rationality of $k(x,y,z)$ over $k$ is equivalent with
the existence of $\mathfrak{G}$-invariant point on $F_\infty$.

If $\deg P$ is odd and $\deg Q$ is even,
then $s=0$ implies $q_0 \in k^2$,
so that $c_0=\sqrt{q_0} \in k$
and $\tau_\sigma$ on $F_\infty$ is given by
$\lambda \mapsto \lambda+\frac{a_0-a_0^\sigma}{2c_0}{p_0}$.
So $\lambda=\frac{a_0p_0}{2c_0}$ is invariant by
$\Phi_\sigma=\tau_\sigma \circ \widetilde{\sigma}$
for all $\sigma \in \mathfrak{G}$.
If both of $\deg P$ and $\deg Q$ are odd,
similar arguments show that
$\lambda=\frac{c_0}{2a_0}$ is $\mathfrak{G}$-invariant.

If $\deg P$ is even and $\deg Q$ is odd,
we have $p_0 \in k^2$ and $\frac{c_0}{a_0} =\sqrt{p_0} \in k$.
$\tau_\sigma$ on $F_\infty$ is given by
$\lambda \mapsto \frac{(c+c^\sigma)\lambda+(a-a^\sigma)p_0}
{(a-a^\sigma)\lambda+(c+c^\sigma)}$,
so $\lambda=\sqrt{p_0}$ is $\mathfrak{G}$-invariant.

Assume that both of $\deg P$ and $\deg Q$ are even.
The rationality of $k(x,y,z)$ is independent of
the choice of $A,B,C$ such that $A^2P+B^2Q=C^2$,
so we can choose convenient ones.

Let $a_0=0, c_0=\sqrt{q_0}$,
then $c_0^\sigma=\pm c_0$.
If $q_0 \in k^2$,
$\tau_\sigma$ on $F_\infty$ is the identity for all $\sigma \in \mathfrak{G}$,
so every $\lambda \in k$ is $\mathfrak{G}$-invariant.
If $\sqrt{q_0} \not\in k$,
every $\mathfrak{G}$-invariant $\lambda$
must belong to $k(c_0)=k(\sqrt{q_0})$.

If $c_0^\sigma=-c_0$, we have $d_{\sigma0}=0$,
so $\tau_\sigma$ is $\lambda \mapsto \frac{p_0}{\lambda}$
on $F_\infty$.
Let $\lambda=\lambda_1+c_0 \lambda_2$, $\lambda_1,\lambda_2 \in k$,
and $\overline{\lambda}=\lambda_1-c_0\lambda_2$.
Then $\lambda$ is $\mathfrak{G}$-invariant
if and only if $\lambda \overline{\lambda}=p_0$,
namely $\lambda_1^2-q_0\lambda_2^2=p_0$.
The existence of such $\lambda_1,\lambda_2$ is equivalent with
the existence of non-zero $(a,b,c) \in k \times k \times k$
such that $a^2p_0+b^2q_0=c^2$.
\end{proof}

\subsection{The case where $s=2$.} \label{s2}

The surface $Y$, on which $\mathfrak{G}$ acts in a Zariski homeomorphic way,
is of type $Y_{r2}$.
$\Pic(Y)$ is of rank 4 with the basis $F,F^\prime,E_1$ and $E_2$.
The intersection form is $F \cdot F=0,
F \cdot F^\prime=1, F^\prime \cdot F^\prime=r$.
$F \cdot E_i=F^\prime \cdot E_i=0,
E_i \cdot E_i=-1 \, (i=1,2)$
and $E_1 \cdot E_2=0$.

First, we shall show that there exists an irreducible curve $\Gamma$
such that $\Gamma \cdot \Gamma<0$ and $\Gamma \cdot F=1$.

Let $Y_{r0}$ be the surface before the blowings-up $E_1$ and $E_2$.
The discussions in the proof of Proposition \ref{prop16} show the followings.

When $r$ is odd, there exists an irreducible curve $\Gamma$
such that $\Gamma \cdot \Gamma <0$ and $\Gamma \cdot F=1$ on $Y_{r0}$.
Since $\Gamma \cdot \Gamma$ decreases by any blowing-up,
we have $\Gamma \cdot \Gamma<0$ on $Y$.

When $r$ is even, there exists a family of irreducible curves
such that $\Gamma \cdot \Gamma=0$ and $\Gamma \cdot F=1$ on $Y_{r0}$.
Among them, there exists a curve $\Gamma$ which passes through $P_1$
(= the base point of $E_1$), and we have $\Gamma \cdot \Gamma<0$ on $Y$.

\begin{prop} \label{prop17}
There exist exactly two irreducible curves
$\Gamma_1$ and $\Gamma_2$ which satisfy $\Gamma_i \cdot \Gamma_i<0$
and $\Gamma_i \cdot F=1 \, (i=1,2)$.
They are mutually conjugate,
and their classes are
\begin{equation}
\Gamma_1: F^\prime-\frac{r}{2}F-E_1, \,
\Gamma_2:F^\prime-\frac{r}{2}F-E_2
\end{equation}
when $r$ is even.

The surface $Y_{r0}$ (before the blowings-up $E_1, E_2$)
is biregular with $\mathbb{P}^1 \times \mathbb{P}^1$.

The odd $r$ case can never happen.
\end{prop}

\begin{proof}
We choose $\sigma \in \mathfrak{G}$ so that it maps $E_1$ to $F-E_1$.
Then it maps $E_2$ to $F-E_2$
(note that $n_\sigma=2$ as stated in Subsection \ref{s1}),
and maps $F^\prime$ to $F^\prime+F-E_1-E_2$
because $F$ and $\Omega=-2F^\prime+(\frac{r}{2}-2)F+E_1+E_2$ are invariant.
So $\sigma$ maps $F^\prime+\nu F$ to $F^\prime+(\nu+1)F-E_1-E_2$ and
maps $F^\prime+\nu F-E_1$ to $F^\prime+\nu F-E_2$.

Suppose that an irreducible curve $\Gamma_1$ belongs to $F^\prime+\nu F-E_1$,
then its conjugate $\Gamma_2$
belongs to $F^\prime+\nu F-E_2$.
Since $0 \leq \Gamma_1 \cdot \Gamma_2=\Gamma_1 \cdot \Gamma_1+1$,
we have $\Gamma_1 \cdot \Gamma_1=-1$
and $\nu=-\frac{r}{2}$, $r$ being even.

Let $v$ be the ratio of the defining equation of $\Gamma_1$ and $\Gamma_2$:
\begin{eqnarray}
v=\frac{f_2(x)u+g_2(x)}{f_1(x)u+g_1(x)}, \\
\Gamma_i: f_i(x)u+g_i(x)=0. \nonumber
\end{eqnarray}
Then, since $\Gamma_1 \cdot \Gamma_2=0$,
$v$ is not $\frac{0}{0}$ at any point of $Y$,
so the birational mapping $(x,u) \mapsto (x,v)$ is regular
from $Y$ to $\mathbb{P}^1 \times \mathbb{P}^1$.
Since it is injective on $Y_{r0}$,
it defines a biregular mapping from $Y_{r0}$ to $\mathbb{P}_1 \times \mathbb{P}_1$.
On the surface $Y$,
there exist two exceptional curves $E_1$ and $E_2$,
and $E_i$ is mapped to $\varepsilon_i$,
where $\varepsilon_1$ is the blowing up at $(c_1, \infty)$
and $\varepsilon_2$ is the blowing up at $(c_2,0)$.

$\sigma$ maps $(v=0)$ to $(v=\infty)$,
so $\tau_\sigma$ is written in terms $v$ as
\begin{equation} \label{eq412}
\tau_\sigma: v \mapsto \frac{\varphi(x)}{v}, \,
\varphi(x) \in l(x).
\end{equation}
But $\tau_\sigma$ is regular on $x \not= c_1,c_2$,
so $\varphi(x)$ has a pole at $c_1$
and a zero at $c_2$,
so that $\varphi=\alpha\frac{x-c_2}{x-c_1}$,
$\alpha \in l$,
thus
\begin{equation} \label{eq413}
\tau_\sigma: v \mapsto \alpha\frac{x-c_2}{x-c_1}\frac{1}{v}
\end{equation}
(if $c_1=\infty$, $\varphi(x)=\alpha(x-c_2)$).

Similar discussion shows that if $r$ were odd,
there would exist two irreducible curves $\Gamma_1$, $\Gamma_2$
whose classes are $F^\prime-\frac{r+1}{2}F$
and $F^\prime-\frac{r-1}{2}F-E_1-E_2$.
Let $v$ be the ratio of the defining equations of $\Gamma_1$ and $\Gamma_2$,
then we would have $(\ref{eq412})$,
but this time $\varphi(x)$ has two zeros and no pole.
Since for any rational function,
the number of zeros are equal with the number of poles (including $x=\infty$),
this is impossible.
Thus $r$ can never be odd.
\end{proof}

Let $\mathfrak{G}_0$ be the subgroup of $\mathfrak{G}$
which acts trivially on $\Pic(Y)$.
Let $\overline{\mathfrak{G}}=\mathfrak{G}/\mathfrak{G}_0$.
Then $\Pic(Y)$ is actually a $\overline{\mathfrak{G}}$-lattice.
Since both of $\Gamma_1, \Gamma_2$ are $\mathfrak{G}_0$-invariant,
$v$ is also $\mathfrak{G}_0$-invariant,
so we can consider only the
$\overline{\mathfrak{G}}$-action on $v$.

\begin{prop}
When $c_1, c_2 \in k \cup \{\infty\}$,
$k(x,y,z)$ is $k$-rational.
\end{prop}

\begin{proof}
In this case,
$c_1$ can never be moved by any $\sigma \in \mathfrak{G}$,
so $|\overline{\mathfrak{G}}|=2$
and the only non-trivial element of $\overline{\mathfrak{G}}$
is the above $\sigma$.
Thus it suffices to find a $\sigma$-invariant transcendental basis.

Obviously $w_1=v+\alpha\frac{x-c_2}{x-c_1}\frac{1}{v}$
and $w_2=\sqrt{\pi_1}\{v-\alpha\frac{x-c_2}{x-c_1}\frac{1}{v}\}$
are $\sigma$-invariant.
Here $\pi_1$ is one of $P(c_1), Q(c_1), -P/Q(c_1),
p_0, q_0$ or $-p_0/q_0$
according to the situation,
and $(x=c_1)$ is an exceptional curve of $\tau_\sigma$
if and only if $\sigma$ maps $\sqrt{\pi_1}$ to $-\sqrt{\pi_1}$.

However, we have $l(x,y,z)=l(x,v)=l(\alpha\frac{x-c_2}{x-c_1},v)
=l(\alpha\frac{x-c_2}{x-c_1}\frac{1}{v},v)=l(w_1,w_2)$.
Thus $(w_1,w_2)$ is a transcendental basis,
and the rationality of $k(x,y,z)$ has been proved.
\end{proof}

\begin{rem}
Both of $C: v_1=$const.
and $C^\prime: v_2=$const.
are in the class $-F-\Omega$ as shown below.

The defining equations of $C$ and $C^\prime$ are linear combinations of
$(x-c_2)f_1f_2$, $(x-c_2)f_1^2$ and $(x-c_1)f_2^2$
whose classes are $2F^\prime+(1-r)F-E_1-2E_2$,
$2F^\prime+(1-r)F-2E_1-E_2$
and $2F^\prime+(1-r)F-E_1-2E_2$ respectively.
So the class of non-trivial linear combination is
$2F^\prime+(1-r)F-E_1-E_2=-\Omega-F$.

From this we see that
$C \cdot C=C^\prime \cdot C^\prime=C \cdot C^\prime=2$,
$\Omega \cdot C=\Omega \cdot C^\prime=-4$.
Consider three point blow-up of $Y$.
If $C \cdot E_1^\prime=C^\prime \cdot E_1^\prime
=C \cdot E_2^\prime=C^\prime \cdot E_3^\prime=1$ and
$C \cdot E_3^\prime=C^\prime \cdot E_2^\prime=0$,
then we reach $C \cdot C=C^\prime \cdot C^\prime=0$,
$C \cdot C^\prime=1$, $\Omega \cdot C=\Omega \cdot C^\prime=-2$
on the blown-up surface $Z^\prime$.
\end{rem}

When $c_1, c_2$ are mutually conjugate (namely when $(x-c_1)(x-c_2)$
is irreducible over $k$),
the analysis is more complicated.
This time, $|\overline{\mathfrak{G}}|=4$
and the non-trivial elements of $\overline{\mathfrak{G}}$
is given as follows:
\begin{eqnarray}
\sigma_1: & E_1 \mapsto F-E_1, & E_2 \mapsto F-E_2, \nonumber \\
\sigma_2: & E_1 \leftrightarrow E_2, & \\
\sigma_3: & E_1 \mapsto F-E_2, & E_2 \mapsto F-E_1. \nonumber
\end{eqnarray}
Note that $n_\sigma$ is even as stated in Subsection \ref{s1}.

The fixed field of $\mathfrak{G}_0$ is $l=k(c,\sqrt{\pi_1})$
and $\Gal(l/k)=\overline{\mathfrak{G}} \simeq C_2 \times C_2$.
Let $k_i$ be the fixed field of $\sigma_i$ ($i=1,2,3$).
They are quadratic extensions of $k$ contained in $l$.
Write $k_i$ as $k_i=k(\sqrt{e_i}), \, e_i \in k$.
Then $e_1=(c_1-c_2)^2$
($=(c_1+c_2)^2-4c_1c_2$).
$l$ is a vector space over $k$
with the basis $1, \sqrt{e_1}, \sqrt{e_2}$ and $\sqrt{e_3}$.
Since $\sigma_1$ and $\sigma_3$ maps $\sqrt{\pi_1}$ to $-\sqrt{\pi_1}$,
we have $\sqrt{\pi_1} \in k\sqrt{e_2}$
so that we can set $e_2=\pi_1$ (especially $\pi_1 \in k$),
and $e_3=e_1e_2$.

By similar discussions as deriving $(\ref{eq413})$ in Proposition \ref{prop17},
we can determine the action of $\overline{\mathfrak{G}}$ on $v$ as follows:
\begin{eqnarray} \label{eq415}
\tau_{\sigma_1}: & v \mapsto \alpha \frac{x-c_2}{x-c_1}\frac{1}{v}, & \alpha \in k_1, \nonumber \\
\tau_{\sigma_2}: & v \mapsto \beta \frac{1}{v}, & \beta \in k_2, \\
\tau_{\sigma_3}: & v \mapsto \gamma \frac{x-c_1}{x-c_2}v, & \gamma\gamma^{\sigma_3}=1. \nonumber
\end{eqnarray}
The problem is to find a $\overline{\mathfrak{G}}$-invariant transcendental basis of $l(x,v)$.
By a suitable constant multiplication of $v$,
we can set $\gamma=1$.
Then from $\sigma_3=\sigma_1\sigma_2=\sigma_2\sigma_1$,
we have $\alpha=\beta \in k$.
So, we set $\alpha=\beta=\kappa \in k$
and $\gamma=1$ in $(\ref{eq415})$.

\begin{prop} \label{prop419}
$k(x,y,z)$ is $k$-rational
if and only if $\kappa \in N_{k_1/k}(k_1)N_{k_2/k}(k_2)$.
\end{prop}

\begin{proof}
We see that $(x-c_1)v$ is $\sigma_3$-invariant
and is mapped to $\kappa\frac{x-c_2}{v}$ by $\sigma_1$ and $\sigma_2$.
So
\begin{eqnarray} \label{eq416}
w_1 &=& (x-c_1)v+\frac{\kappa(x-c_2)}{v}, \\
w_2 &=& \frac{1}{\sqrt{e_3}} \{(x-c_1)v-\frac{\kappa(x-c_2)}{v}\} \nonumber
\end{eqnarray}
are $\overline{\mathfrak{G}}$-invariant.
Note that $\sqrt{e_3}$ is $\sigma_3$-invariant
and is mapped to $-\sqrt{e_3}$ by $\sigma_1$ and $\sigma_2$.

We shall eliminate $v$ from $(\ref{eq416})$,
then we obtain $w_1^2-e_3w_2^2=\kappa(x^{\prime2}-e_1)$
where $x^\prime=2x-c_1-c_2$.
From this we see that $k(x^\prime,w_1,w_2)$
is rational if and only if the corresponding quadratic form has
non-zero solution in $k$,
namely if and only if
$\kappa \in N_{k_1/k}(k_1)N_{k_2/k}(k_2)$.
\end{proof}

\begin{prop}
$k(x,y,z)$ is $k$-rational
if and only if it is $k_2$-rational, where $k_2=k(\sqrt{\pi_1})$.
\end{prop}

\begin{rem}
Since $s=0$ over $k_2$,
$k(x,y,z)$ is $k_2$-rational except the following case.

Both of $\deg P$ and $\deg Q$ are even
and $a^2p_0+b^2q_0=c^2$ has non-zero solution in $k(\sqrt{\pi_1})$
(note that the above $k_2$ is denoted by $k_1$
in Theorem \ref{mainthm}).
\end{rem}

\begin{proof}
The action of $\sigma_2$ depends only on $v$,
independent of $x$,
so $k(x,y,z)$ is $k_2$-rational if and only if
$\beta=\lambda\lambda^{\sigma_2}$
for some $\lambda \in l$.
This means that by a suitable constant multiplication of $v$,
we can set $\beta=1$.

When $\beta=1$,
from $\sigma_3=\sigma_1\sigma_2=\sigma_2\sigma_1$,
we have $\gamma=\frac{1}{\alpha}=\alpha^{\sigma_3}$.
Then by a multiplication of $v$ by $\gamma+1$,
$\gamma$ is reduced to $1$
and $\beta=1$ is reduced to
$\kappa=N_{k_1/k}(\frac{1}{\alpha}+1) \in N_{k_1/k}(k_1)$.
The proof is completed by virtue of Proposition \ref{prop419}.
\end{proof}

\subsection{The case where $s=3$.} \label{s3}

Similar discussions as in the proof of Proposition \ref{prop17}
show the following:

\begin{prop} \label{props3a}
There exist four irreducible curves $\Gamma$
such that $\Gamma \cdot \Gamma<0$ and $\Gamma \cdot F=1$.
When $r$ is even, they are
$\Gamma_1 \equiv F^\prime -\frac{r}{2}F-E_1$,
$\Gamma_2 \equiv F^\prime-\frac{r}{2}F-E_2$,
$\Gamma_3 \equiv F^\prime-\frac{r}{2}F-E_3$ and
$\Gamma_4 \equiv F^\prime+(1-\frac{r}{2})F-E_1-E_2-E_3$.

When $r$ is odd, they are
$\Gamma_1 \equiv F^\prime-\frac{r-1}{2}F-E_2-E_3$,
$\Gamma_2 \equiv F^\prime-\frac{r-1}{2}F-E_1-E_3$,
$\Gamma_3 \equiv F^\prime-\frac{r-1}{2}F-E_1-E_2$ and
$\Gamma_4 \equiv F^\prime-\frac{r+1}{2}F$.

They are mutually conjugate by the action of $\overline{\mathfrak{G}}$
and satisfy $\Gamma_i \cdot \Gamma_i=\Gamma_i \cdot \Omega=-1$
and $\Gamma_i \cdot \Gamma_j=0 \, (i \not= j)$.
\end{prop}

The action of $\overline{\mathfrak{G}}$ induces a permutation of $\Gamma_i$,
we can blow down $Y$ by $\Gamma_i$
and $\overline{\mathfrak{G}}$ still acts on the blown down surface $Y^\prime$
in a Zariski homeomorphic way.
Since $\Omega_Y^\prime \cdot \Omega_Y^\prime=5+4=9$,
$Y^\prime$ is biregular with the projective plane $\mathbb{P}^2$.
So $\overline{\mathfrak{G}}$ acts on $\mathbb{P}^2$ in a Zariski homeomorphic way,
but a biregular transformation of $\mathbb{P}^2$
is nothing but a linear transformation of the homogeneous coordinate
$(\xi,\eta,\zeta)$, since $\Aut(\mathbb{P}^2) \simeq PGL(3,\overline{k})$.

Each of $\Gamma_i$ is mapped to a point $P_j$ on $\mathbb{P}^2$ $(1 \leq j \leq 4)$.
By checking the intersection form,
 we see that $\{P_j\}$ is not colinear
(= \, any three points do not lie on the same line),
so that we can choose $P_1$ as $(1,0,0)$,
$P_2$ as $(0,1,0)$, $P_3$ as $(0,0,1)$ and $P_4$ as $(1,1,1)$.

\begin{prop}
When $s=3$,
$k(x,y,z)$ is $k$-rational.
\end{prop}

\begin{proof}
We shall divide the proof into three subcases.

\hfill\break
$(\mathrm{I})$
The case where all $c_i \in k$ or $\infty$ $(1 \leq i \leq 3)$.

In this case, $|\overline{\mathfrak{G}}|=4$
and three non-trivial elements of $\overline{\mathfrak{G}}$ are as follows:
\begin{eqnarray*}
\sigma_1: & E_1 \mapsto E_1, E_2 \mapsto F-E_2, E_3 \mapsto F-E_3, \\
\sigma_2: & E_1 \mapsto F-E_1, E_2 \mapsto E_2, E_3 \mapsto F-E_3, \\
\sigma_3: & E_1 \mapsto F-E_1, E_2 \mapsto F-E_2, E_3 \mapsto E_3.
\end{eqnarray*}
The action of $\sigma_1$ switches $\Gamma_1$ and $\Gamma_4$,
and $\Gamma_2$ and $\Gamma_3$ (whether $r$ is even or odd),
so it switches $P_1$ and $P_4$, and $P_2$ and $P_3$ on $\mathbb{P}^2$.
Namely $\sigma_1$ induces the permutation $(1 4)(2 3)$ of $\{P_j\}$.
Similarly $\sigma_2$ induces $(2 4)(1 3)$ and $\sigma_3$ induces $(3 4)(1 2)$.
The biregular transformation of $\mathbb{P}^2$
which induces the above permutation of $\{P_j\}$ is described as follows:
$$\sigma_1=\begin{pmatrix}
-1 & 0 & 0 \\
-1 & 0 & 1 \\
-1 & 1 & 0
\end{pmatrix},
\sigma_2=\begin{pmatrix}
0 & -1 & 1 \\
0 & -1 & 0 \\
1 & -1 & 0
\end{pmatrix},
\sigma_3=\begin{pmatrix}
0 & 1 & -1 \\
1 & 0 & -1 \\
0 & 0 & -1
\end{pmatrix}.$$
Namely, the action of $\sigma_1$ is
$\xi \mapsto -\xi, \eta \mapsto \zeta-\xi, \zeta \mapsto \eta-\xi$, etc.

Put $\xi^\prime=-\xi+\eta+\zeta, \eta^\prime=\xi-\eta+\zeta,
\zeta^\prime=\xi+\eta-\zeta$.
Then $\xi^\prime$ is mapped to $\xi^\prime$ by $\sigma_1$
and mapped to $-\xi$ by $\sigma_2$ and $\sigma_3$.
The similar holds for $\eta^\prime$ and $\zeta^\prime$ also.

On the other hand, let $\pi_i$ be one of
$P(c_i), Q(c_i), -\frac{Q}{P}(c_i), p_0, q_0, -\frac{q_0}{p_0}$,
according to the situation.
Then $\pi_i \in k \setminus k^2$
and $E_i$ is mapped to $x=c_i$ if and only if
$\sqrt{\pi_i}^\sigma=-\sqrt{\pi_i}$.
So $\sqrt{\pi_1}$ is mapped to $\sqrt{\pi_1}$ by $\sigma_1$
and mapped to $-\sqrt{\pi_1}$ by $\sigma_2$ and $\sigma_3$.
Hence $\sqrt{\pi_1} \xi^\prime$ is
invariant by all $\sigma \in \overline{\mathfrak{G}}$.

Similarly $\sqrt{\pi_2}\eta^\prime$ and $\sqrt{\pi_3}\zeta^\prime$
are also $\overline{\mathfrak{G}}$-invariant.
Thus, $\xi^{\prime\prime}=\sqrt{\pi_1}\xi^\prime$,
$\eta^{\prime\prime}=\sqrt{\pi_2}\eta^\prime$,
$\zeta^{\prime\prime}=\sqrt{\pi_3}\zeta^\prime$
become a $\mathfrak{G}$-invariant homogeneous coordinate of $\mathbb{P}^2$,
so they yield a $\mathfrak{G}$-invariant transcendent basis
$(v_1,v_2)$ of $\overline{k}(x,u)$ by
$v_1=\xi^{\prime\prime}/\zeta^{\prime\prime}$, $v_2=\eta^{\prime\prime}/\zeta^{\prime\prime}$.

\noindent
$(\mathrm{II})$
The case where $c_1$ and $c_2$ are conjugates,
and $c_3 \in k$ or $\infty$.

Let $\overline{\mathfrak{G}}_1$ be the subgroup of $\overline{\mathfrak{G}}$,
trivial on $k(c_1)=k(c_2)$.
We have $[\overline{\mathfrak{G}}:\overline{\mathfrak{G}}_1]=2$,
and $\overline{\mathfrak{G}}_1$ acts as mentioned in $(\mathrm{I})$.
But if $\sqrt{\pi_3} \in k(c_1)$,
$\sqrt{\pi_3}^\sigma=-\sqrt{\pi_3}$ is impossible for $\sigma \in \overline{\mathfrak{G}}_1$,
so $|\overline{\mathfrak{G}}_1|=2, |\overline{\mathfrak{G}}|=4$.
If $\sqrt{\pi_3} \not\in k(c_1)$,
then $|\overline{\mathfrak{G}}_1|=4, |\overline{\mathfrak{G}}|=8$.

\noindent
$(\mathrm{II-1})$
When $\sqrt{\pi_3} \not\in k(c_1)$.

As shown in $(\mathrm{I})$, $\sqrt{\pi_1}\xi^\prime, \sqrt{\pi_2}\eta^\prime, \sqrt{\pi_3}\zeta^\prime$
are $\overline{\mathfrak{G}}_1$-invariant.
$\overline{\mathfrak{G}} \simeq D_4$ and contains
$\sigma_4: E_1 \mapsto E_2, E_2 \mapsto E_1, E_3 \mapsto E_3$,
which induces the permutation $(1 2)$ of $\{P_j\}$.
The corresponding biregular transformation of $\mathbb{P}^2$ is described by
$\begin{pmatrix}
0 & 1 & 0 \\
1 & 0 & 0 \\
0 & 0 & 1
\end{pmatrix}$,
namely $\sigma_4$ maps as $\xi \rightarrow \eta, \eta \mapsto \xi, \zeta \mapsto \zeta$.

$\sigma_4$ maps as $\xi^\prime \rightarrow \eta^\prime, \eta^\prime \rightarrow \xi^\prime,
\zeta^\prime \rightarrow \zeta^\prime$.
On the other hand, since $\sigma_4 \not\in \overline{\mathfrak{G}}_1$
and is of order $2$,
$\sigma_4$ maps $\sqrt{\pi_1}$ to $\sqrt{\pi_2}$, $\sqrt{\pi_2}$ to $\sqrt{\pi_1}$, or
$\sqrt{\pi_1}$to $-\sqrt{\pi_2}$, $\sqrt{\pi_2}$ to $-\sqrt{\pi_1}$.

Suppose that $\sqrt{\pi_1} \mapsto \sqrt{\pi_2}, \sqrt{\pi_2} \mapsto \sqrt{\pi_1}$.
Then the action of $\sigma_4$ switches
$\sqrt{\pi_1}\xi^\prime$ and $\sqrt{\pi_2}\eta^\prime$,
so $\sqrt{\pi_1}\xi^\prime+\sqrt{\pi_2}\eta^\prime$ and
$(c_1-c_2)(\sqrt{\pi_1}\xi^\prime-\sqrt{\pi_2}\eta^\prime)$ are
$\overline{\mathfrak{G}}$-invariant
(If $\sqrt{\pi_1} \mapsto -\sqrt{\pi_2}, \sqrt{\pi_2} \mapsto -\sqrt{\pi_1}$,
then it suffices only to replace $\sqrt{\pi_2}$ with $-\sqrt{\pi_2}$).

\noindent
$(\mathrm{II-2})$
When $\sqrt{\pi_3} \in k(c_1)$.

In this case, $\overline{\mathfrak{G}}$ contains $\sigma_5:E_1 \mapsto E_2, E_2 \mapsto F-E_1,
E_3 \mapsto F-E_3$, so that
$\overline{\mathfrak{G}}$ is a cyclic group of order 4.

Since this $\overline{\mathfrak{G}}$ is a proper subgroup of $\overline{\mathfrak{G}}$ in $(\mathrm{II-1})$,
the basis obtained in $(\mathrm{II-1})$ is $\overline{\mathfrak{G}}$-invariant in this case also.
Namely $\xi^{\prime\prime}=\sqrt{\pi_1}\xi^\prime+\sqrt{\pi_2}\eta^\prime$,
$\eta^{\prime\prime}=(c_1-c_2)(\sqrt{\pi_1}\xi^\prime-\sqrt{\pi_2}\eta^\prime)$,
$\zeta^{\prime\prime}=\sqrt{\pi_3}\zeta^\prime$
become a $\overline{\mathfrak{G}}$-invariant homogeneous coordinate of $\mathbb{P}^2$, and
$v_1=\xi^{\prime\prime}/\zeta^{\prime\prime}$, $v_2=\eta^{\prime\prime}/\zeta^{\prime\prime}$.
is a $\mathfrak{G}$-invariant transcendent basis of $\overline{k}(x,u)$.

\noindent
$(\mathrm{III})$
The case where all $c_i$ are mutually conjugate $(1 \leq i \leq 3)$.

Let $K=k(c_1,c_2,c_3)$ be the smallest decomposition field.
Let $\overline{\mathfrak{G}}_1$ be the subgroup of $\overline{\mathfrak{G}}$,
trivial on $K$.

\noindent
$(\mathrm{III-1})$
When $[K:k]=6$ and $\sqrt{\pi_i} \not\in K (1 \leq i \leq 3)$.

Then $\overline{\mathfrak{G}}_1$ is $\overline{\mathfrak{G}}$ in $(\mathrm{I})$,
and $\overline{\mathfrak{G}}/\overline{\mathfrak{G}}_1 \simeq \mathfrak{S}_3$.
(So $\overline{\mathfrak{G}} \simeq \mathfrak{S}_4$).
$\sqrt{\pi_1}\xi^\prime, \sqrt{\pi_2}\eta^\prime, \sqrt{\pi_3}\zeta^\prime$ are
$\overline{\mathfrak{G}}_1$-invariant
as proved in $(\mathrm{I})$.
A permutation of $E_i$ induces a permutation of $P_i (1 \leq i \leq 3)$,
so induces a permutation of $\xi^\prime, \eta^\prime, \zeta^\prime$.
Then $\sqrt{\pi_1}, \sqrt{\pi_2}, \sqrt{\pi_3}$ are also permuted
in the same way.
So $\sqrt{\pi_1}\xi^\prime+\sqrt{\pi_2}\eta^\prime+\sqrt{\pi_3}\zeta^\prime$
is $\overline{\mathfrak{G}}$-invariant.
Similarly $c_1\sqrt{\pi_1}\xi^\prime+c_2\sqrt{\pi_2}\eta^\prime+c_3\sqrt{\pi_3}\zeta^\prime$, and
$c_1^2\sqrt{\pi_1}\xi^\prime+c_2^2\sqrt{\pi_2}\eta^\prime+c_3^2\sqrt{\pi_3}\zeta^\prime$
are $\overline{\mathfrak{G}}$-invariant.

\noindent
$(\mathrm{III-2})$
When $[K:k]=3$.

The proof is same as in $(\mathrm{III-1})$, except $\mathfrak{S}_3$ is replaced by $C_3$.
(So $\overline{\mathfrak{G}} \simeq \mathfrak{A}_4$).

\noindent
$(\mathrm{III-3})$
When $[K:k]=6$ and $\sqrt{\pi_i} \in K$.

Then $|\overline{\mathfrak{G}}|=6, \overline{\mathfrak{G}} \simeq \mathfrak{S}_3$.
$\sigma_5$ in $(\mathrm{II-2})$ is contained in $\overline{\mathfrak{G}}$.
The order of $\sigma_5$ is $4$.
This contradicts with $|\overline{\mathfrak{G}}|=6$.
Therefore this case can never happen
(under the assumption $s=3$).

\end{proof}

\begin{rem}
Both of $C: v_1=$const.
and $C^\prime: v_2=$const.
are in the class $-F-\Omega$ as shown below.

Let $h_i(x,u)=f_i(x)u+g_i(x)$ be the defining polynomials of $\Gamma_i$.

When $r$ is even, $\xi=0$ on $\mathbb{P}^2$ decomposes into
three curves $x=c_1, \Gamma_2, \Gamma_3$ on $Y$,
therefore $\xi=a(x-c_1)h_2h_3$ with some $a \in \overline{k}$.
Similarly $\eta=b(x-c_2)h_1h_3, \zeta=c(x-c_3)h_1h_2$.
The defining equations of $C$ and $C^\prime$ are linear combinations of
$(x-c_1)h_2h_3, (x-c_2)h_1h_3, (x-c_3)h_1h_2$
whose class is all $2F^\prime+(1-r)F-E_1-E_2-E_3=-\Omega-F$.

When $r$ is odd, $\xi=0$ on $\mathbb{P}^2$
decomposes into three curves $E_1, \Gamma_2, \Gamma_3$ on $Y$,
therefore $\xi=ah_2h_3$.
Similarly $\eta=bh_1h_3, \zeta=ch_1h_2$.
The defining equations of $C$ and $C^\prime$ are linear combinations of
$h_2h_3, h_1h_3, h_1h_2$ whose classes are
$2F^\prime+(1-r)F-2E_1-E_2-E_3, 2F^\prime+(1-r)F-E_1-2E_2-E_3,
2F^\prime+(1-r)F-E_1-E_2-2E_3$ respectively.
A non-trivial linear combination is in the class $-\Omega-F$.
\end{rem}

From this we see that $C \cdot C=C^\prime \cdot C^\prime=C \cdot C^\prime=1$,
$\Omega \cdot C=\Omega \cdot C^\prime=-3$.
Consider two point blow-up of $Y$.
If $C \cdot E_1^\prime=C^\prime \cdot E_2^\prime=1$
and $C \cdot E_2^\prime=C^\prime \cdot E_1^\prime=0$,
then we reach $C \cdot C=C^\prime \cdot C^\prime=0,
C \cdot C^\prime=1, \Omega \cdot C=\Omega \cdot C^\prime=-2$
on the blown up surface $Z^\prime$.

\section*{Acknowledgement}
The author would like to thank
Ming-chang Kang who gave many valuable advices and comments.
He also would like to thank J.-L. Colliot-Th\'el\`ene who gave useful comment
for Theorem \ref{A}.

\end{document}